\documentclass[12pt]{amsart}

\usepackage{amsmath,amssymb,ifthen}
\usepackage{fullpage}
\usepackage{graphicx,psfrag,subfigure}
\usepackage{color}
\usepackage{moreverb}

\DeclareMathOperator*{\argmin}{arg\,min}
\def\R{{\mathbb R}}
\def\N{{\mathbb N}}

\def\NN{{\mathcal N}}
\def\HH{{\mathcal H}}
\def\OO{{\mathcal O}}

\def\PPrec{{\mathbf P}}

\def\TT{{\mathcal T}}
\def\XX{{\mathcal X}}
\def\YY{{\mathcal Y}}
\def\diam{{\rm diam}}

\def\norm#1#2{\|#1\|_{#2}}

\def\set#1#2{\big\{#1\,:\,#2\big\}}

\def\dual#1#2{\langle#1\,,\,#2\rangle}

\def\uu{\mathbf{u}}
\def\UU{\mathbf{U}}
\def\uuu{\boldsymbol{u}}
\def\vvv{\boldsymbol{v}}
\def\vv{\mathbf{v}}
\def\VV{\mathbf{V}}
\def\xx{\boldsymbol{X}}
\def\yy{\mathbf{Y}}

\def\diag{{\rm diag}}

\def\div{{\rm div}}
\def\normal{\boldsymbol{n}}

\def\slp{\mathcal{V}} 
\def\dlp{\mathcal{K}} 
\def\hyp{\mathcal{W}} 
\def\jn{\mathcal{L}} 
\def\sym{{\widetilde{\jn}}} 
\def\symRHS{{\widetilde{F}}}
\def\bmcRHS{{\widetilde{F}}}
\def\bmc{{\widetilde{\jn}}} 

\def\PAShyp{\mathfrak{W}_{\rm AS}}
\def\PASslp{\mathfrak{V}_{\rm AS}}
\def\prechyp{{\mathfrak{W}}}
\def\precfem{{\mathfrak{A}}}
\def\PASfem{\precfem_{\rm AS}}
\def\precslp{{\mathfrak{V}}}
\def\pform{\mathcal{B}} 
\def\tmp{\mathcal{Q}} 
\def\tmpmat{\mathbf{Q}}
\def\fem{\mathcal{A}}

\def\hmax{h_{\rm max}}
\def\hmin{h_{\rm min}}
\def\evmax{\lambda_{\rm max}}
\def\evmin{\lambda_{\rm min}}
\def\cond{{\rm cond}}
\def\embedding{{\mathcal I}}
\def\embmat{{\mathbf I}}
\def\DD{{\mathbf D}}
\def\material{A}
\def\Haar{{\mathbf H}}
\def\ZZ{{\mathcal Z}}
\def\embedP{\mathcal{J}}
\def\embedPmat{\mathbf{J}}
\def\onemat{\mathbf{1}}

\def\Omegaext{\Omega^{\rm ext}}

\def\RR{\mathbf{R}}   

\def\bignull{\boldsymbol{0}} 
\def\linhull{{\rm span}} 

\def\pphi{\boldsymbol{\Phi}}

\def\GG{\mathbf{G}}

\def\AA{\mathbf{A}}
\def\FF{\mathbf{F}}
\def\MM{\mathcal{M}}
\def\Mmat{\mathbf{M}}
\def\stabmat{\mathbf{S}}
\def\stabmatBmc{{\widetilde{\stabmat}}}
\def\Kmat{\mathbf{K}}

\def\ww{\boldsymbol{w}}

\def\WW{\mathbf{W}}
\def\ee{\mathbf{E}}

\renewcommand{\vector}[3]{\begin{pmatrix} #1 \\ #2 \ifthenelse{\equal{#3}{}}{}{\\#3}\end{pmatrix}}
\def\refine{{\rm refine}}

\newcounter{constantsnumber}
\def\namec#1#2{%
  \ifthenelse{\equal{#1}{lipschitz}}{C_{\rm lip}}{%
  \ifthenelse{\equal{#1}{gmres}}{C_{\rm GMRES}}{%
  \ifthenelse{\equal{#1}{monotone}}{C_{\rm mon}}{%
  \ifthenelse{\equal{#1}{cea}}{C_{\mbox{\rm\scriptsize C\'ea}}}{%
  \ifthenelse{\equal{#1}{norm}}{C_{\rm norm}}{%
  \ifthenelse{\equal{#1}{mon}}{{C}_{\rm mon}}{
  \ifthenelse{\equal{#1}{lip}}{{C}_{\rm lip}}{
  \ifthenelse{\equal{#1}{monA}}{c_{\rm mon}}{
  \ifthenelse{\equal{#1}{lipA}}{c_{\rm lip}}{
  \ifthenelse{\equal{#1}{normequiv1}}{c_{\rm norm}}{ 
  \ifthenelse{\equal{#1}{inv}}{C_{\rm inv}}{ 
  \ifthenelse{\equal{#1}{inv2}}{\widetilde{C}_{\rm inv}}{ 
  \ifthenelse{\equal{#2}{newcounter}}{\refstepcounter{constantsnumber}\label{const#1}}{}C_{\ref{const#1}}}%
  }}}}}}}}}}}}
\def\setc#1{\namec{#1}{newcounter}}
\def\c#1{\namec{#1}{reference}}

\newtheorem{theorem}{Theorem}
\newtheorem{proposition}[theorem]{Proposition}
\newtheorem{lemma}[theorem]{Lemma}

\newtheorem{algorithm}[theorem]{Algorithm}

\newtheorem{remark}[theorem]{Remark}

\def\subsection#1
{
 \bigskip

 \refstepcounter{subsection}
 {\noindent\bf\arabic{section}.\arabic{subsection}.~#1.~}
}

\begin{document}

\title[Optimal preconditioning for the symmetric and non-symmetric coupling of
adaptive finite elements and boundary elements]{Optimal preconditioning for the symmetric and non-symmetric coupling of
adaptive finite elements and boundary elements}
\date{\today}

\author{M.~Feischl}
\author{T.~F\"uhrer}
\author{D.~Praetorius}
\author{E.P.~Stephan}
\address{Institute for Analysis and Scientific Computing,
       Vienna University of Technology,
       Wiedner Hauptstra\ss{}e 8--10,
       A-1040 Wien, Austria}
\email{Thomas.Fuehrer@tuwien.ac.at {\it(corresponding author)}}
\email{\{Michael.Feischl,\,Dirk.Praetorius\}@tuwien.ac.at}
\address{Institute for Applied Mathematics,
       Leibniz University Hannover,
       Welfengarten 1,
       D-30167 Hannover, Germany}
\email{stephan@ifam.uni-hannover.de}

\thanks{{\bf Acknowledgment:}
  The research of the authors MF, TF, and DP is supported through the FWF research
  project \textit{Adaptive Boundary Element Method}, see {\tt
  http://www.asc.tuwien.ac.at/abem/}, funded by the Austrian Science Fund (FWF)
  under grant P21732, as well as through the \textit{Innovative Projects
  Initiative} of Vienna University of Technology. 
  Moreover, MF and DP are partially funded through the FWF doctoral school 
  \emph{Dissipation and Dispersion in Nonlinear PDEs}, funded under grant 
  W1245. 
  This support is thankfully
  acknowledged.
}
\begin{abstract}
We analyze a multilevel diagonal additive Schwarz preconditioner for
the adaptive coupling of FEM and BEM for a linear 2D Laplace transmission
problem. We rigorously prove that the condition number of the 
preconditioned system stays uniformly bounded, independently of the
refinement level and the local mesh-size of the underlying adaptively
refined triangulations. Although the focus is on the non-symmetric 
Johnson-N\'ed\'elec one-equation coupling, the principle ideas also apply
to other formulations like the symmetric FEM-BEM coupling.
Numerical experiments underline our theoretical findings.
\end{abstract}
\maketitle

\section{Introduction}
There exist plenty of works on preconditioning of FEM-BEM coupling equations,
covering mainly the symmetric coupling with quasi-uniform meshes,
see~\cite{cckuhnlanger,fs09,harbrecht03,hms99,heuerstep98,kuhnsteinbach,ms98} and the
references therein.
In contrast to that, only little is known on preconditioning of the non-symmetric
Johnson-N\'ed\'elec coupling, see e.g.~\cite{meddahi98}, and also on
preconditioning of \emph{adaptive} FEM-BEM couplings.
It is the main goal of this paper to close this gap and to extend the existing
analysis to the case of the (adaptive) symmetric as well as non-symmetric Johnson-N\'ed\'elec
coupling~\cite{johned}.
For the symmetric coupling~\cite{costabel,han90}, the approach of Bramble \&
Pasciak~\cite{brampasc88} applies, which guarantees
positive definiteness and symmetry of the Galerkin matrix with respect to a special inner
product~\cite{cckuhnlanger,harbrecht03,kuhnsteinbach}.
Therefore, efficient iterative solvers designed for symmetric and positive
definite matrices are applicable.
However, due to the non-symmetry, such an approach may not work for the
Johnson-N\'ed\'elec coupling in general.
In~\cite{meddahi98}, it was assumed that the coupling boundary is smooth. Hence,
the double-layer integral operator $\dlp$ is compact. 
The system matrix can therefore be split into a symmetric part plus a compact
perturbation part $\Kmat$ (the Galerkin matrix of the double-layer integral operator).
Preconditioning is done only on the symmetric part with the theory
of~\cite{brampasc88}, and convergence results for iterative solvers can then be obtained by compact
perturbation theory assuming that the mesh-size of the coarsest mesh is
sufficiently small.
In general, however, the coupling boundary is not smooth. Therefore, the preconditioner
theory must not rely on compactness of $\dlp$. 

To state the contributions of the current work, we consider the non-symmetric stiffness matrix of the (stabilized) 
Johnson-N\'ed\'elec coupling, which reads in block-form
\begin{align}\label{eq:intro:jnmat}
  \AA_\jn := \begin{pmatrix}
    \AA_\material & -\Mmat^T \\
    \tfrac12 \Mmat-\Kmat & \AA_\slp
  \end{pmatrix} + \stabmat\stabmat^T  \in \R^{(N+M)\times (N+M)},
\end{align}
see Section~\ref{sec:main} below.
Here, $\stabmat \in\R^{N+M}$
denotes an appropriate stabilization vector, which ensures positive definiteness of
$\AA_\jn$. The $N\times N$ matrix block $\AA_\material$ is the (positive semi-definite) Galerkin matrix of
the FEM part, and the $M\times M$ matrix block $\AA_\slp$ is
the Galerkin matrix of the simple-layer integral operator $\slp$.
As in~\cite{fs09,ms98}, we deal with block-diagonal preconditioners of the form
\begin{align}\label{eq:intro:prec}
  \PPrec_\jn = \begin{pmatrix}
    \PPrec_\fem & \bignull \\
    \bignull & \PPrec_\slp
  \end{pmatrix}.
\end{align}
Here, the appropriate operator $\fem : H^1(\Omega)\to (H^1(\Omega))^*$ induces a coercive, symmetric, and
bounded bilinear form $\dual{\fem(\cdot)}{(\cdot)} \simeq
\norm\cdot{H^1(\Omega)}^2$.
The symmetric and positive definite matrices $\PPrec_\fem$ resp. $\PPrec_\slp$
are spectrally equivalent to the
symmetric and positive definite Galerkin matrices $\AA_\fem$ resp. $\AA_\slp$, i.e.
\begin{subequations}\label{eq:intro:specest}
\begin{align}
  d_\fem \dual{\PPrec_\fem\xx}\xx_2 &\leq \dual{\AA_\fem\xx}\xx_2 \leq D_\fem
  \dual{\PPrec_\fem\xx}\xx_2 \quad\text{for all }\xx\in\R^N, \\
   d_\slp \dual{\PPrec_\slp\pphi}\pphi_2 &\leq \dual{\AA_\slp\pphi}\pphi_2 \leq D_\slp
  \dual{\PPrec_\slp\pphi}\pphi_2 \quad\text{for all }\pphi\in\R^M,
\end{align}
\end{subequations}
where $\AA_\fem$ is strongly related to the FEM block $\AA_\material$ of the
FEM-BEM system~\eqref{eq:intro:jnmat},
see~\eqref{eq:galerkin:matrices}--\eqref{eq:femoperator} below.
Inspired by~\cite{ms98}, we prove that the condition number of $\PPrec_\jn^{-1}\AA_\jn$ as well as the
number of iterations to reduce the relative residual by a factor $\tau$ in the
preconditioned GMRES algorithm with inner product $\dual\cdot\cdot_{\PPrec_\jn}$ depends
only on $\max\{D_\fem,D_\slp\} / \min\{d_\fem,d_\slp\}$.

Usually, the condition number of Galerkin matrices $\AA_\fem$
and $\AA_\slp$ on adaptively refined meshes hinges
on the global mesh-ratio as well as on the number of degrees of freedom.
Therefore, the construction of optimal preconditioners for iterative
solvers is a necessary task.
Here, optimality is understood in the sense that the condition number of the
preconditioned matrix is independent of the mesh-size and the degrees of
freedom.

Very recently, it was proven in~\cite{xch10} for 2D FEM with energy
space $H^1$ that a local multilevel
additive Schwarz preconditioner $\PPrec_\fem$ is optimal, i.e. the constants $d_\fem,D_\fem$
in~\eqref{eq:intro:specest}
are independent of the mesh-size, the number of degrees of freedom, and the number of levels.
Here, ``local'' means, that scaling at each level is done only on newly created
nodes plus neighbouring nodes, where the associated basis functions have
changed.
An analogous result for 2D and 3D hypersingular integral equations with energy
space $H^{1/2}$ has been derived by the authors~\cite{ffps}.
In~\cite{ffps,xch10}, the proofs rely on a stable space decomposition
of the discrete subspaces in $H^1$ resp. $H^{1/2}$.
Alternatively,~\cite{xcn09} provides stable subspace decompositions in $H^1$ for
higher order elements in any dimension on bisection grids.

In the present work, we prove the optimality of some local multilevel additive
Schwarz preconditioner $\PPrec_\slp$ for 2D weakly-singular integral equations
with energy space $H^{-1/2}$.
The proof is derived by postprocessing of the corresponding result for the
hypersingular integral equation~\cite{ffps}. Combining this with the result of~\cite{xch10},
we prove optimality of $\PPrec_\jn$ for the FEM-BEM coupling.

{\bf Notation.} Throughout the work, we explicitly state all constants and their
dependencies in all statements of results. In proofs, however, we use the notation
$a\lesssim b$ to abbreviate $a\leq C b$ with a constant $C>0$ which is clear
from the context. Moreover, $a\simeq b$ abbreviates $a\lesssim b\lesssim a$. Furthermore, the entries of a vector $\mathbf b$
or a matrix $\AA$ are denoted by $(\mathbf b)_j$ resp. $(\AA)_{jk}$.
By $\dual\cdot\cdot$, we denote the duality brackets between a
Hilbert space $\HH$ and its dual $\HH^*$.

{\bf Outline.} The remainder of this work is organized as follows:
In Section~\ref{sec:main}, we recall the basic facts on the Johnson-N\'ed\'elec
coupling and define the admissible mesh-refinement strategies. 
We also formulate the preconditioned GMRES Algorithm~\ref{alg:gmres}, which is
required to state the main result (Theorem~\ref{thm:main}).
In Section~\ref{sec:spectral:estimates}, we prove the spectral
estimates~\eqref{eq:intro:specest}.
Section~\ref{sec:proof} adapts the analysis for the symmetric
coupling~\cite{ms98} and contains the proof of the main result
(Theorem~\ref{thm:main}).
The short Section~\ref{sec:extensions} deals with extensions of the developed theory to the
symmetric coupling and the one-equation Bielak-MacCamy coupling.
Numerical examples from Section~\ref{sec:examples} underline our theoretical
predictions and conclude the work.


\section{Johnson-N\'ed\'elec coupling and main result}\label{sec:main}

\subsection{Model problem and analytical setting}
Let $\Omega\subset \R^2$ be a polygonal and simply connected domain 
with boundary $\Gamma = \partial\Omega$.
We consider the following Laplace transmission problem in free space:
\begin{subequations}\label{eq:model}
\begin{align}
  -\div(\material\nabla u) &= f &&\text{in } \Omega, \\
  -\Delta u^{\rm ext} &= 0 &&\text{in }\Omegaext := \R^2
  \backslash\overline\Omega, \label{jn:ext}\\
  u- u^{\rm ext} &= u_0 &&\text{on }\Gamma, \\
  (\material\nabla u - \nabla u^{\rm ext})\cdot \normal &= \phi_0
  &&\text{on }\Gamma, \\
  |u^{\rm ext}(x)| &= \OO(1/|x|) && \text{as } |x|\to\infty. \label{eq:model:rad}
\end{align}
\end{subequations}
Here, $\normal$ denotes the outer normal on $\Gamma$, and 
$\material\in L^\infty(\Omega)$ satisfies
$\material(x) \in \R_{\rm sym}^{2\times 2}$ with uniform bounds on the
maximal resp.\ minimal eigenvalue
\begin{align}\label{eq:const:material}
 0 < c_\material := {\rm ess}\inf_{\hspace*{-6mm}x\in\Omega} \evmin(\material(x)) 
  \le {\rm ess}\sup_{\hspace*{-6mm}x\in\Omega} \evmax(\material(x)) 
  =: C_\material < \infty.
\end{align}
With $H^1(\Omega)$ resp. $H^{1/2}(\Gamma)$ and its dual
$H^{-1/2}(\Gamma)=H^{1/2}(\Gamma)^*$, we denote the usual Sobolev spaces
on $\Omega$ resp. $\Gamma$.
For given data $f\in L^2(\Omega)$, $u_0\in H^{1/2}(\Gamma)$, and $\phi_0 \in
H^{-1/2}(\Gamma)$, it is well-known that the model problem~\eqref{eq:model} admits a unique
solution $u\in H^1(\Omega)$, $u^{\rm ext} \in H_{\rm loc}^1(\Omegaext)$
with finite energy $\nabla u^{\rm ext}\in L^2(\Omegaext)$, if we impose the compatibility condition
\begin{align}
  \dual{f}1_\Omega + \dual{\phi_0}1_\Gamma = 0
\end{align}
to ensure the radiation condition~\eqref{eq:model:rad}. Here, $\dual\cdot\cdot_\Omega$ stands for the
$L^2(\Omega)$ inner product, whereas $\dual\cdot\cdot_\Gamma$ denotes the
extended $L^2(\Gamma)$ inner product.
\subsection{Johnson-N\'ed\'elec coupling}
For the formulation of the Johnson-N\'ed\'elec coupling~\cite{johned}, the
exterior solution~\eqref{jn:ext} is formulated by Green's third formula.
The latter gives rise to the simple-layer and double-layer integral operator 
\begin{subequations}\label{def:bio}
\begin{align}\label{def:slp}
 \slp\in L(H^{-1/2+s}(\Gamma);H^{1/2+s}(\Gamma)),\quad
 &\slp\phi(x) := \int_\Gamma\,G(x,y)\phi(y)\,dy,\\
 \label{def:dlp}
 \dlp\in L(H^{1/2+s}(\Gamma);H^{1/2+s}(\Gamma)),\quad
 &\dlp v(x) := \int_\Gamma\partial_{\normal(y)}G(x,y)\,v(y)\,dy,
\end{align}
where $G(x,y) := -\frac{1}{2\pi}\,\log|x-y|$ denotes the fundamental solution
of the 2D Laplacian and $\partial_{\normal}(\cdot)$ is the normal derivative. 
Note that boundedness holds for all $-1/2\le s\le1/2$.
In addition, our analysis requires the hypersingular integral operator
\begin{align}\label{def:hyp}
 \hyp\in L(H^{1/2+s}(\Gamma);H^{-1/2+s}(\Gamma)),\quad
 &\hyp v(x) := -\partial_{\normal(x)}\int_\Gamma\partial_{\normal(y)}G(x,y)\,v(y)\,dy,
\end{align}
\end{subequations}
where the integral is understood as finite part integral. It is known that
$\slp$ and $\hyp$ are symmetric in the sense of 
\begin{align}\label{eq:bio:symmetry}
 \dual{\phi}{\slp\psi}_\Gamma 
 = \dual{\psi}{\slp\phi}_\Gamma
 \quad\text{and}\quad
 \dual{\hyp v}{w}_\Gamma
 = \dual{\hyp w}{v}_\Gamma
\end{align}
for all $\phi,\psi\in H^{-1/2}(\Gamma)$ and 
$v,w\in H^{1/2}(\Gamma)$.
Moreover, the assumption
$\diam(\Omega)<1$ ensures $H^{-1/2}(\Gamma)$-ellipticity of $\slp$,
\begin{align}\label{eq:slp:ellipticity}
 c_\slp\,\norm{\psi}{H^{-1/2}(\Gamma)}^2
 \le \dual{\psi}{\slp\psi}_\Gamma
 \quad\text{for all }\psi\in H^{-1/2}(\Gamma).
\end{align}
Finally, $\hyp$ is semi-elliptic with kernel being the constant functions,
\begin{align}\label{eq:hyp:semielliptic}
 \hyp1 = 0 
 \quad\text{and}\quad
 c_\hyp\,\norm{v}{H^{1/2}(\Gamma)}^2
 \le \dual{\hyp v}{v}_\Gamma
 \quad\text{for all }v\in H^{1/2}(\Gamma)
 \text{ with }\int_\Gamma v\,dx = 0.
\end{align}
The constants $c_\slp,c_\hyp>0$ depend only on $\Omega$.

With the definitions~\eqref{def:bio} of the layer integral operators, the
model problem~\eqref{eq:model} is equivalently recast by the 
Johnson-N\'ed\'elec coupling:
Given $f\in L^2(\Omega)$, $u_0 \in H^{1/2}(\Gamma)$, and $\phi_0 \in H^{-1/2}(\Gamma)$, find $(u,\phi) \in
\HH:= H^1(\Omega) \times H^{-1/2}(\Gamma)$ such that
\begin{subequations}\label{eq:jn}
  \begin{align}
    \dual{\material \nabla u}{\nabla v}_\Omega - \dual\phi{v}_\Gamma &= \dual{f}v_\Omega +
    \dual{\phi_0}v_\Gamma, \\
    \dual\psi{(\tfrac12-\dlp)u + \slp\phi}_\Gamma &=
    \dual\psi{(\tfrac12-\dlp)u_0}_\Gamma
  \end{align}
\end{subequations}
for all $(v,\psi)\in\HH$. 
Setting $(u,\phi)=(1,0)=(v,\psi)$ in~\eqref{eq:jn}, we see that the 
(non-stabilized) linear
operator associated to the left-hand side of~\eqref{eq:jn} is indefinite.
However, let $1/2 \leq c_\dlp < 1$ denote the contraction constant of the 
double-layer potential~\cite{sw}. Following the analysis 
in~\cite{os,sayas09,s2}, also (stabilized) Galerkin formulations 
of~\eqref{eq:jn} admit unique solutions, if the ellipticity constant
$c_\material$ from~\eqref{eq:const:material} satisfies 
\begin{align}\label{eq:contraction}
  c_\material \geq c_\dlp / 4
\end{align}
and if the equation is either \emph{explicitly stabilized}~\cite{os,s2} or if the
discrete subspace of $H^{-1/2}(\Gamma)$ contains the constant functions~\cite{sayas09}.
In~\cite{affkmp}, the result of~\cite{sayas09} is reproduced with a new proof.
Introducing the notion of \emph{implicit stabilization}, 
an equivalent elliptic operator equation of~\eqref{eq:jn} is derived
which fits in the frame of the Lax-Milgram lemma and 
thus leads to non-symmetric, but positive definite Galerkin 
matrices. The main result of~\cite{affkmp} reads as follows
(and also holds for a strongly monotone, but nonlinear material tensor $A$):

\begin{lemma}\label{lem:jn:stab}
Let~\eqref{eq:contraction} be satisfied. For $(u,\phi),(v,\psi)\in\HH=\XX\times\YY$, define
\begin{subequations}\label{eq:jn:staboperator}
  \begin{align}
    \begin{split}
    \dual{\jn(u,\phi)}{(v,\psi)} &:= \dual{\material \nabla u}{\nabla v}_\Omega - 
    \dual\phi{v}_\Gamma + \dual\psi{(\tfrac12-\dlp)u + \slp\phi}_\Gamma   \\
    &\qquad +\dual{1}{(\tfrac12-\dlp)u+\slp\phi}_\Gamma\dual{1}{(\tfrac12-\dlp)v+\slp\psi}_\Gamma
  \end{split}
  \end{align}
  as well as 
  \begin{align}\begin{split}
    \dual{F}{(v,\psi)} &:= \dual{f}v_\Omega + \dual{\phi_0}v_\Gamma + \dual{\psi}{(\tfrac12-\dlp)u_0}_\Gamma 
    \\ &\qquad + \dual{1}{(\tfrac12-\dlp)u_0}_\Gamma\dual{1}{(\tfrac12-\dlp)v+\slp\psi}_\Gamma.
  \end{split}
  \end{align}
\end{subequations}
  Let $\XX^\ell\subseteq\XX$ and $\YY^\ell\subseteq\YY$ be arbitrary closed
  subspaces with $1\in\YY^\ell$. Then, the pair $(u_\ell,\phi_\ell)\in\HH^\ell
  := \XX^\ell\times\YY^\ell$ solves the Galerkin formulation of the
  Johnson-N\'ed\'elec coupling~\eqref{eq:jn} 
  \begin{subequations}\label{eq:jn:disrete}
    \begin{align}
      \dual{\material \nabla u_\ell}{\nabla v_\ell}_\Omega -
      \dual{\phi_\ell}{v_\ell}_\Gamma &= \dual{f}{v_\ell}_\Omega +
      \dual{\phi_0}{v_\ell}_\Gamma, \\
      \dual{\psi_\ell}{(\tfrac12-\dlp)u_\ell + \slp\phi_\ell}_\Gamma &=
      \dual{\psi_\ell}{(\tfrac12-\dlp)u_0}_\Gamma
    \end{align}
  \end{subequations} 
  for all $(v_\ell,\psi_\ell)\in\HH^\ell$, if and only if it solves 
  the operator formulation
  \begin{align}\label{eq:varform:stabilized}
  \begin{split}
    \dual{\jn(u_\ell,\phi_\ell)}{(v_\ell,\psi_\ell)} &= \dual{F}{(v_\ell,\psi_\ell)}
  \end{split}
  \end{align}
  for all $(v_\ell,\psi_\ell)\in\HH^\ell$.
  The operator $\jn:\HH\to\HH^*$ is non-symmetric, but linear, continuous,
  and elliptic, and the constants
  \begin{align}
    c_\jn := \inf_{\bignull\neq (u,\phi)\in\HH}\frac{\dual{\jn(u,\phi)}{(u,\phi)}}{\norm{(u,\phi)}\HH^2}
    \quad\text{and}\quad
    C_\jn := \sup_{\bignull\neq (u,\phi)\in\HH}\frac{\norm{\jn(u,\phi)}{\HH^*}}{\norm{(u,\phi)}\HH}
  \end{align}
  satisfy $0<c_\jn\le C_\jn<\infty$ and depend only on $c_\material$ and $C_\material$ from~\eqref{eq:const:material} as well as on $\Omega$. Moreover, $F\in\HH^*$ is a continuous linear functional on $\HH$.
  In particular,~\eqref{eq:varform:stabilized} (and hence
  also~\eqref{eq:jn:disrete}) admits a unique solution
  $(u_\ell,\phi_\ell)\in\HH^\ell$, and there holds the C\'ea-type estimate
  \begin{align}
    \norm{(u,\phi)-(u_\ell,\phi_\ell)}\HH \leq \frac{C_\jn}{c_\jn}
    \inf\limits_{(v_\ell,\psi_\ell)\in\HH^\ell}
    \norm{(u,\phi)-(v_\ell,\psi_\ell)}\HH,
  \end{align}
  where $(u,\phi)\in\HH$ denotes the unique solution of the 
  Johnson-N\'ed\'elec coupling~\eqref{eq:jn}.\qed
\end{lemma}

\subsection{Adaptive mesh-refinement and discrete spaces}\label{sec:refinement}
\begin{figure}[t]
 \centering
 \psfrag{T0}{}
 \psfrag{T1}{}
 \psfrag{T2}{}
 \psfrag{T3}{}
 \psfrag{T4}{}
 \psfrag{T12}{}
 \psfrag{T34}{}
 \includegraphics[width=35mm]{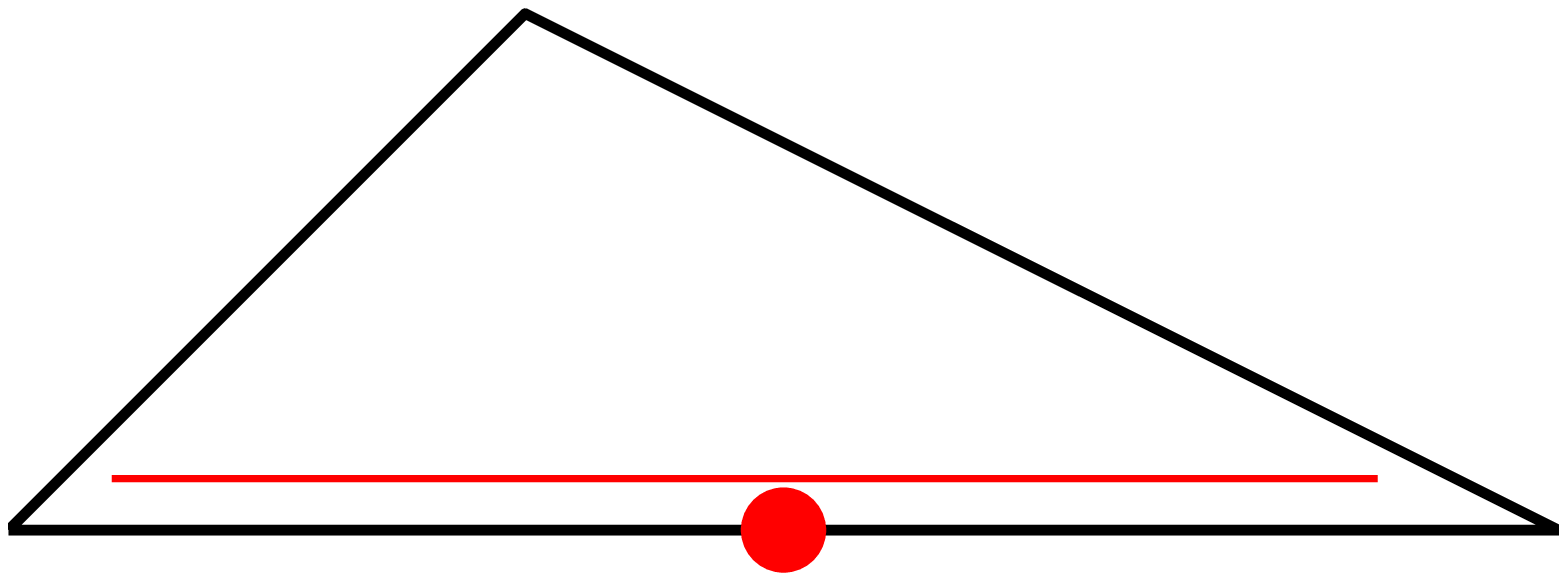} \quad
 \includegraphics[width=35mm]{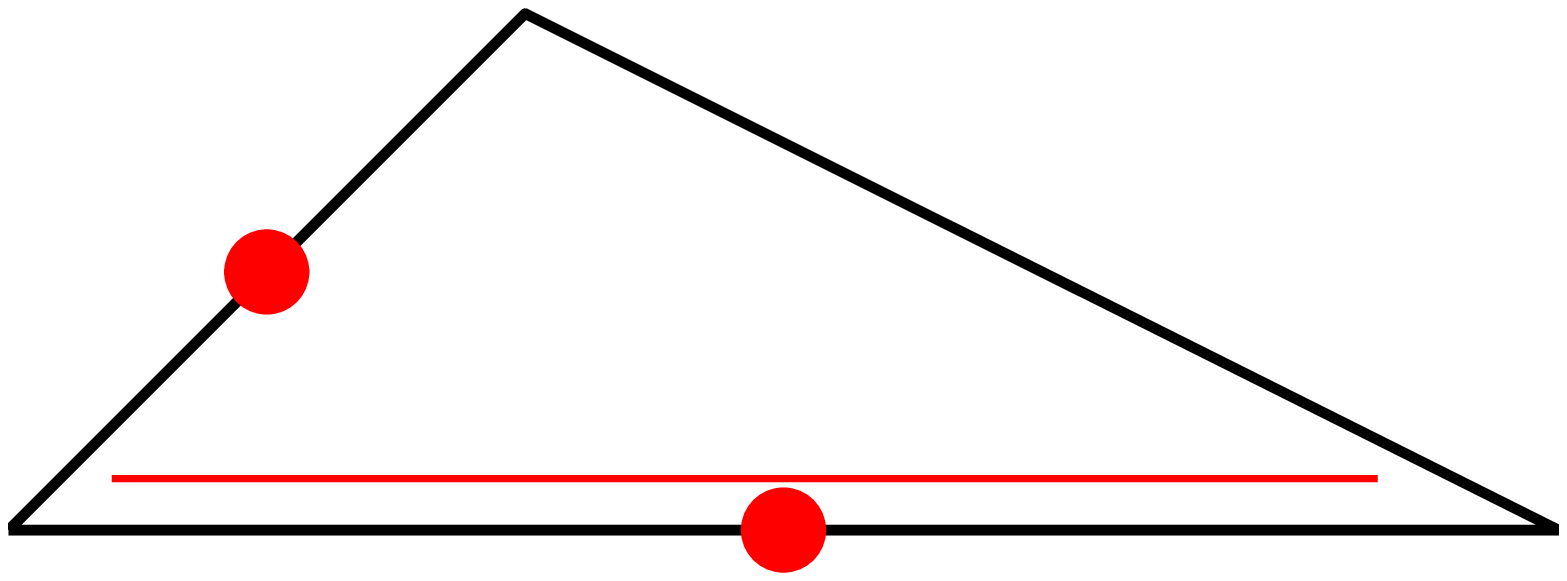} \quad
 \includegraphics[width=35mm]{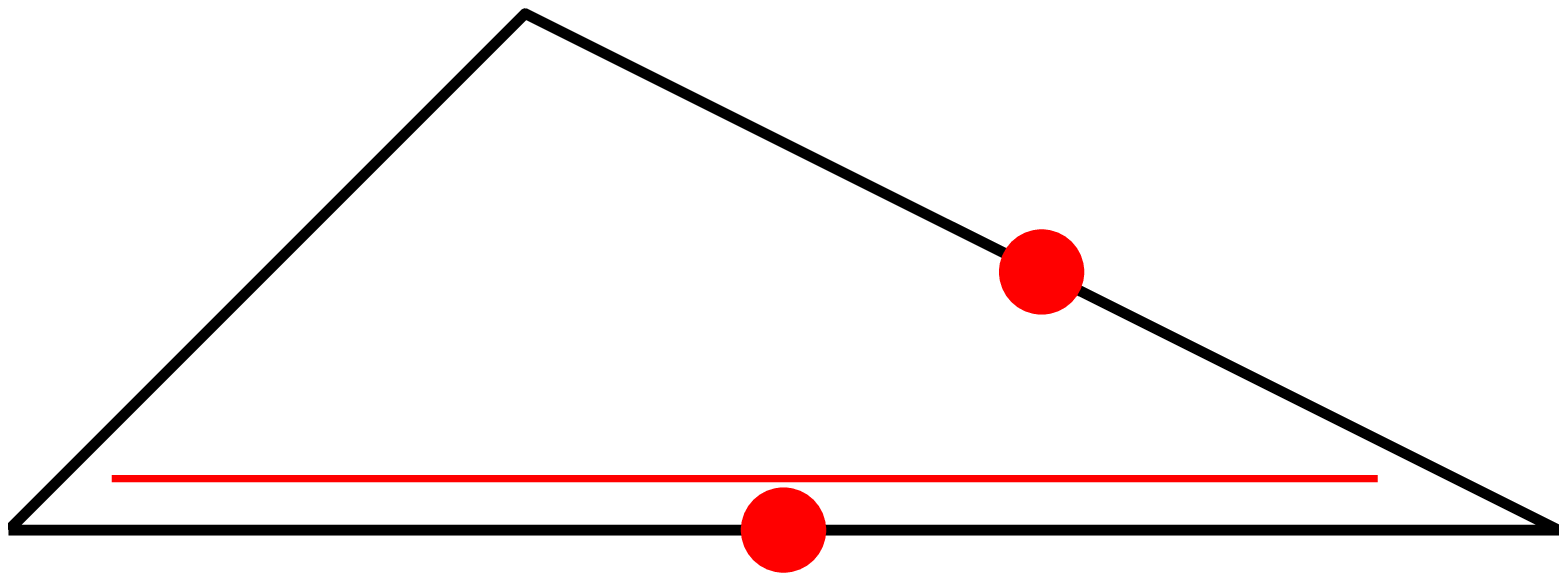} \quad
 \includegraphics[width=35mm]{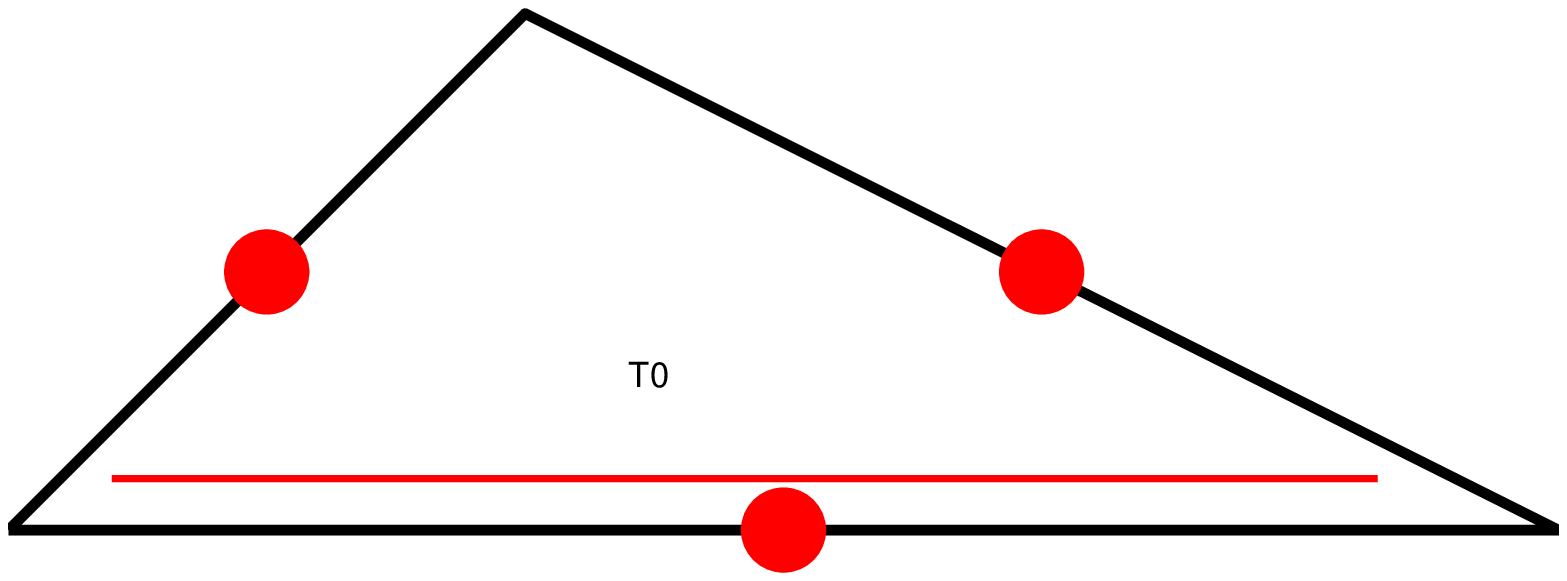} \\
 \includegraphics[width=35mm]{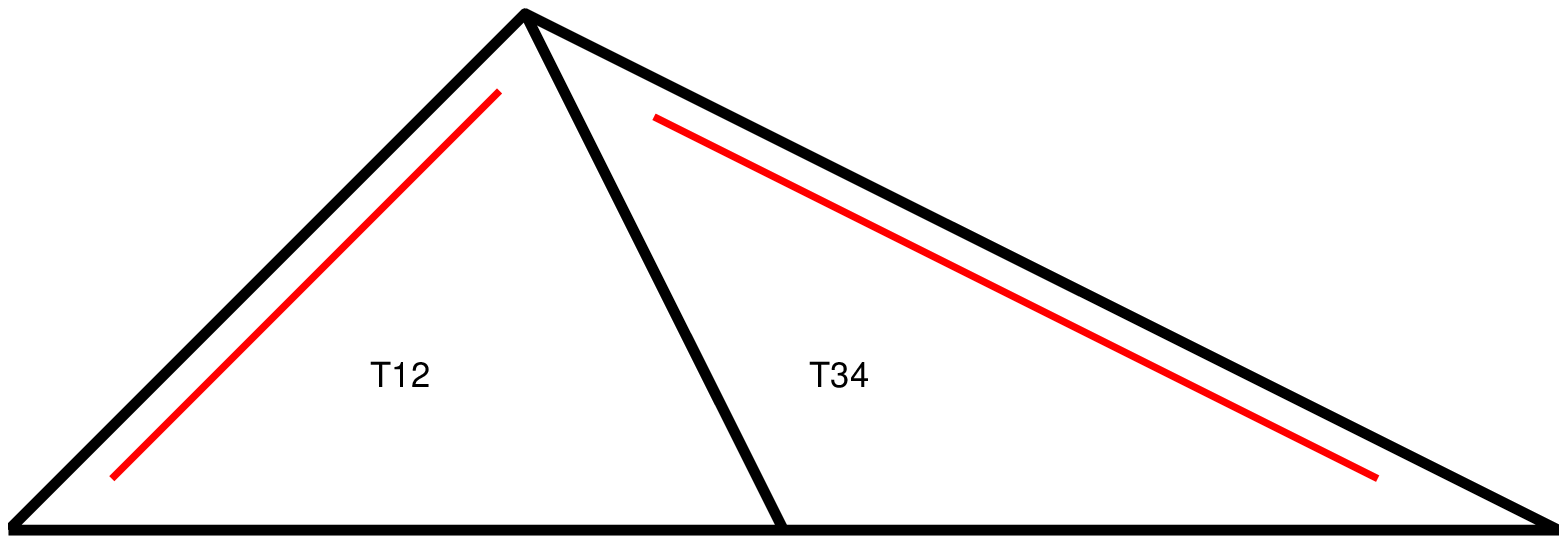} \quad
 \includegraphics[width=35mm]{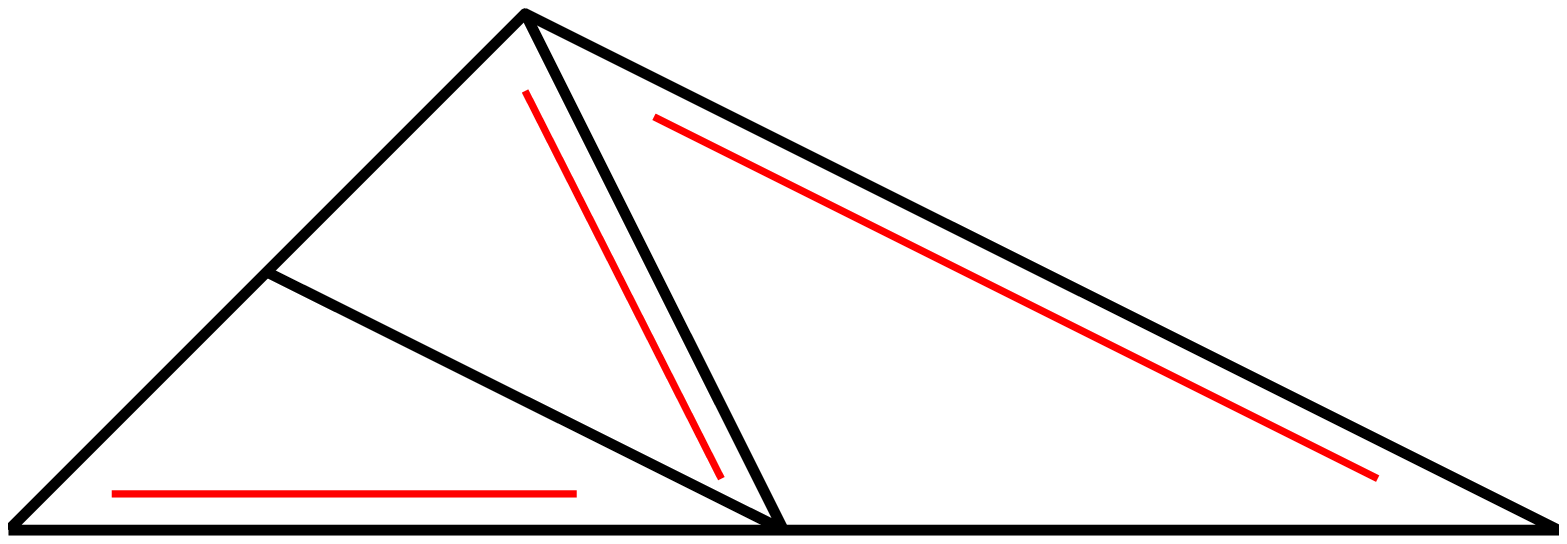}\quad
 \includegraphics[width=35mm]{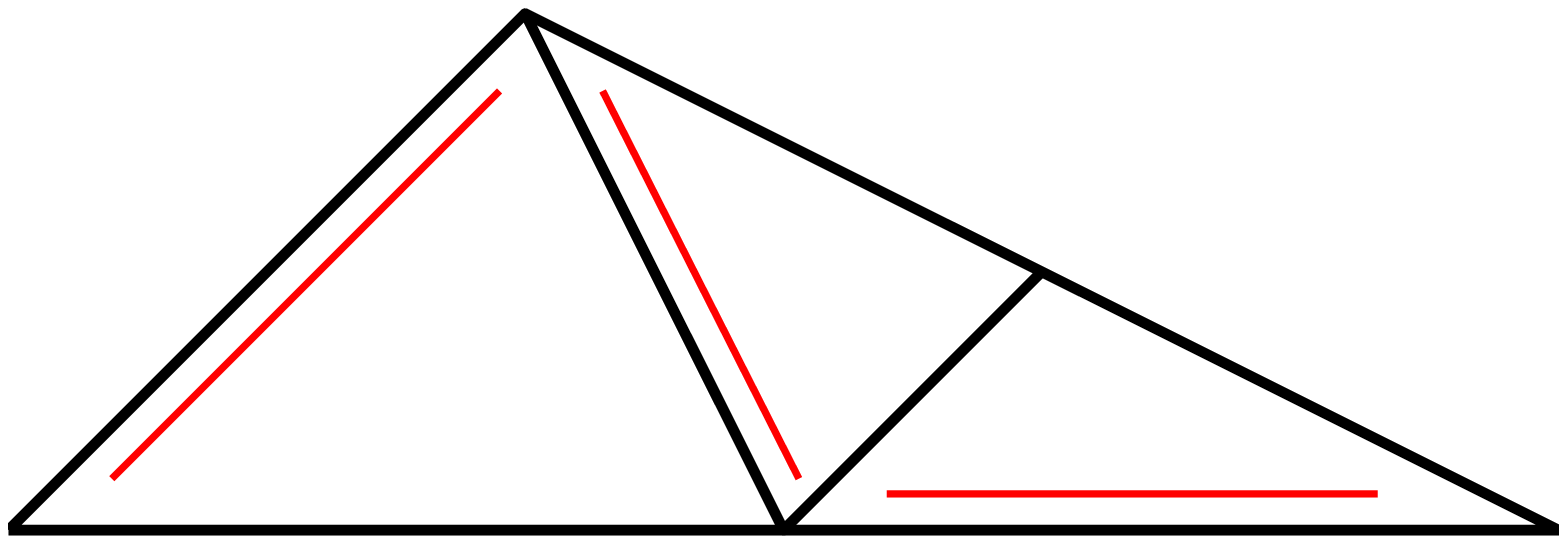}\quad
 \includegraphics[width=35mm]{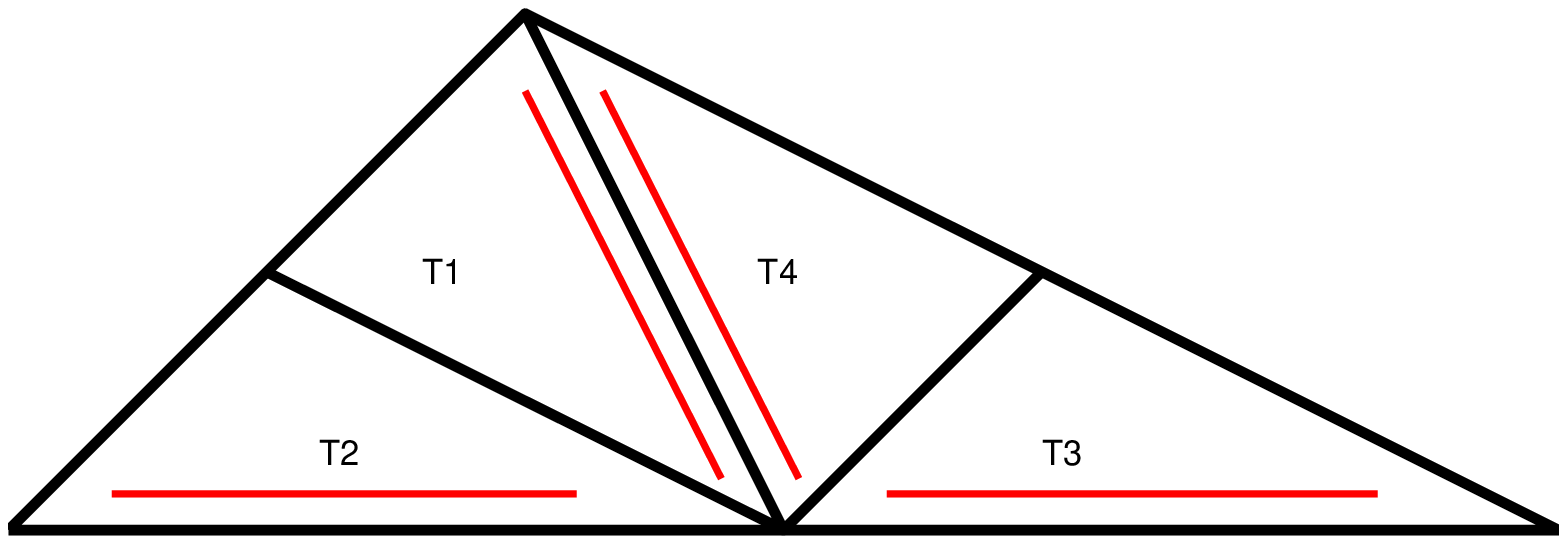}
 \caption{
 For each triangle $T\in\TT_\ell^\Omega$, there is one fixed 
 \emph{reference edge},
 indicated by the double line (left, top). Refinement of $T$ is done by bisecting
 the reference edge, where its midpoint becomes a new node. The reference
 edges of the son triangles are opposite to this newest vertex (left, bottom).
 To avoid hanging nodes, one proceeds as follows:
 We assume that certain edges of $T$, but at least the reference edge,
 are marked for refinement (top).
 Using iterated newest vertex bisection, the element is then split into
 2, 3, or 4 son triangles (bottom).}
 \label{fig:nvb:bisec}
\end{figure}
Let $\TT_0^\Omega$ be a given conforming initial triangulation of $\Omega$ into
compact and non-degenerate triangles. We suppose that a sequence
$\TT_{\ell+1}^\Omega = \refine(\TT_\ell^\Omega,\MM_\ell^\Omega)$ of refined
triangulations is obtained by newest vertex bisection, see
Figure~\ref{fig:nvb:bisec}, where $\TT_{\ell+1}^\Omega$ is the coarsest conforming
mesh such that all marked elements $T\in\MM_\ell^\Omega \subseteq
\TT_\ell^\Omega$ have been bisected. For our analysis, the set $\emptyset\neq\MM_\ell^\Omega
\subseteq \TT_\ell^\Omega$ of marked elements is arbitrary, but in
practice obtained from local a posteriori refinement indicators, see
e.g.~\cite{affkmp}. We note that newest vertex bisection guarantees uniform
shape regularity in the sense that
\begin{align}
  \sup\limits_{T\in\TT_\ell^\Omega} \frac{\diam(T)}{|T|^{1/2}} \leq \gamma
  <\infty \quad\text{for all } \ell\in\N_0,
\end{align}
where $\gamma$ depends only on the initial mesh $\TT_0^\Omega$, see
e.g.~\cite{verfuerth,kpp} and the references therein.

Let $\TT_0^\Gamma$ be a given initial partition of the coupling boundary
$\Gamma$ into compact line segments. We suppose that a sequence
$\TT_{\ell+1}^\Gamma = \refine(\TT_\ell^\Gamma,\MM_\ell^\Gamma)$ of refined
partitions is obtained by bisection, where the refined elements
$T\in\TT_\ell^\Gamma\backslash\TT_{\ell+1}^\Gamma$ are refined into two sons
$T',T''\in\TT_{\ell+1}^\Gamma\backslash\TT_\ell^\Gamma$ of half length, i.e.
$T=T'\cup T''$ with $\diam(T')=\diam(T'')=\diam(T)/2$ and where at least the
marked elements $T\in\MM_\ell^\Gamma \subseteq \TT_\ell^\Gamma$ are
refined, i.e.
$\MM_\ell^\Gamma \subseteq \TT_\ell^\Gamma \backslash \TT_{\ell+1}^\Gamma$.
In addition, we suppose that the meshes are uniformly $\gamma$-shape regular in
the sense that
\begin{align}
  \frac{\diam(T)}{\diam(T')} \leq \gamma < \infty \quad\text{for all }
  T,T'\in\TT_\ell^\Gamma \text{ with } T\cap T' \neq \emptyset \text{ and all }
  \ell\in\N_0,
\end{align}
where $\gamma$ depends only on the initial partition $\TT_0^\Gamma$. Possible
choices include the 1D bisection algorithms from~\cite{cmam}. A further choice
is to consider the partition $\TT_\ell^\Gamma := \TT_\ell^\Omega|_\Gamma$ of
$\Gamma$ which is induced by the triangulation $\TT_\ell^\Omega$ of $\Omega$.
Formally, such a coupling of $\TT_\ell^\Omega$ and $\TT_\ell^\Gamma$ is not
required for the analysis, but simplifies the implementation and is therefore
used in the numerical experiments of Section~\ref{sec:examples}.

In this work, we consider lowest-order Galerkin elements. We approximate
functions $u\in H^1(\Omega)$ by functions $u_\ell \in \XX^\ell$ and functions $\phi\in
H^{-1/2}(\Gamma)$ by functions $\phi_\ell\in\YY^\ell$, where
\begin{align}
  \XX^\ell &:= \set{v\in C(\Omega)}{v|_T \text{ is affine for all } T\in\TT_\ell^\Omega}, \\
  \YY^\ell &:= \set{\psi\in L^2(\Gamma)}{\psi|_T \text{ is
  constant for all } T\in\TT_\ell^\Gamma}.
\end{align}

Let $\NN_\ell^\Omega$ denote the set of nodes of the triangulation
$\TT_\ell^\Omega$, and let
$\NN_\ell^\Gamma$ denote the set of nodes of the triangulation $\TT_\ell^\Gamma$.
For $z\in\NN_\ell^\Omega$ resp.\ $z\in\NN_\ell^\Gamma$, we define the patch
$\omega_\ell^\Omega(z) := \set{T \in \TT_\ell^\Omega}{ z \in T}$ resp.
$\omega_\ell^\Gamma(z) := \set{T \in \TT_\ell^\Gamma}{ z \in T}$.
For the construction of optimal multilevel preconditioners on adaptively refined
triangulations, we need the following subsets of $\NN_\ell^\Omega$ resp.\ 
$\NN_\ell^\Gamma$:
\begin{subequations}\label{def:Ntilde}
\begin{align}
  \widetilde\NN_0^\Omega &:= \NN_0^\Omega,
  &\widetilde\NN_\ell^\Omega &:= \NN_\ell^\Omega \backslash \NN_{\ell-1}^\Omega \cup 
  \set{z\in\NN_{\ell-1}^\Omega}{\omega_\ell^\Omega(z) \subsetneqq
  \omega_{\ell-1}^\Omega(z)} \quad\text{for }\ell\geq 1, \\
  \widetilde\NN_0^\Gamma &:= \NN_0^\Gamma
  &\widetilde\NN_\ell^\Gamma &:= \NN_\ell^\Gamma \backslash \NN_{\ell-1}^\Gamma \cup 
  \set{z\in\NN_{\ell-1}^\Gamma}{\omega_\ell^\Gamma(z) \subsetneqq
  \omega_{\ell-1}^\Gamma(z)} \quad\text{for }\ell\geq 1.
\end{align}
\end{subequations}
The sets $\widetilde\NN_\ell^\Omega$ resp. $\widetilde\NN_\ell^\Gamma$ thus
consist of the new nodes plus the old nodes, where the corresponding patches have
changed.
A visualization of $\widetilde\NN_\ell^\Omega$ is given in
Figure~\ref{fig:Ntilde}.
For each node $z\in\NN_\ell^\Omega$, $\eta_z^\ell \in\XX^\ell$ denotes the associated hat-function
with $\eta_z^\ell(z') = \delta_{zz'}$ for all $z'\in\NN_\ell^\Omega$, where
$\delta_{zz'}$ is Kronecker's delta.

\begin{figure}[t]
 \centering
 \includegraphics[width=65mm]{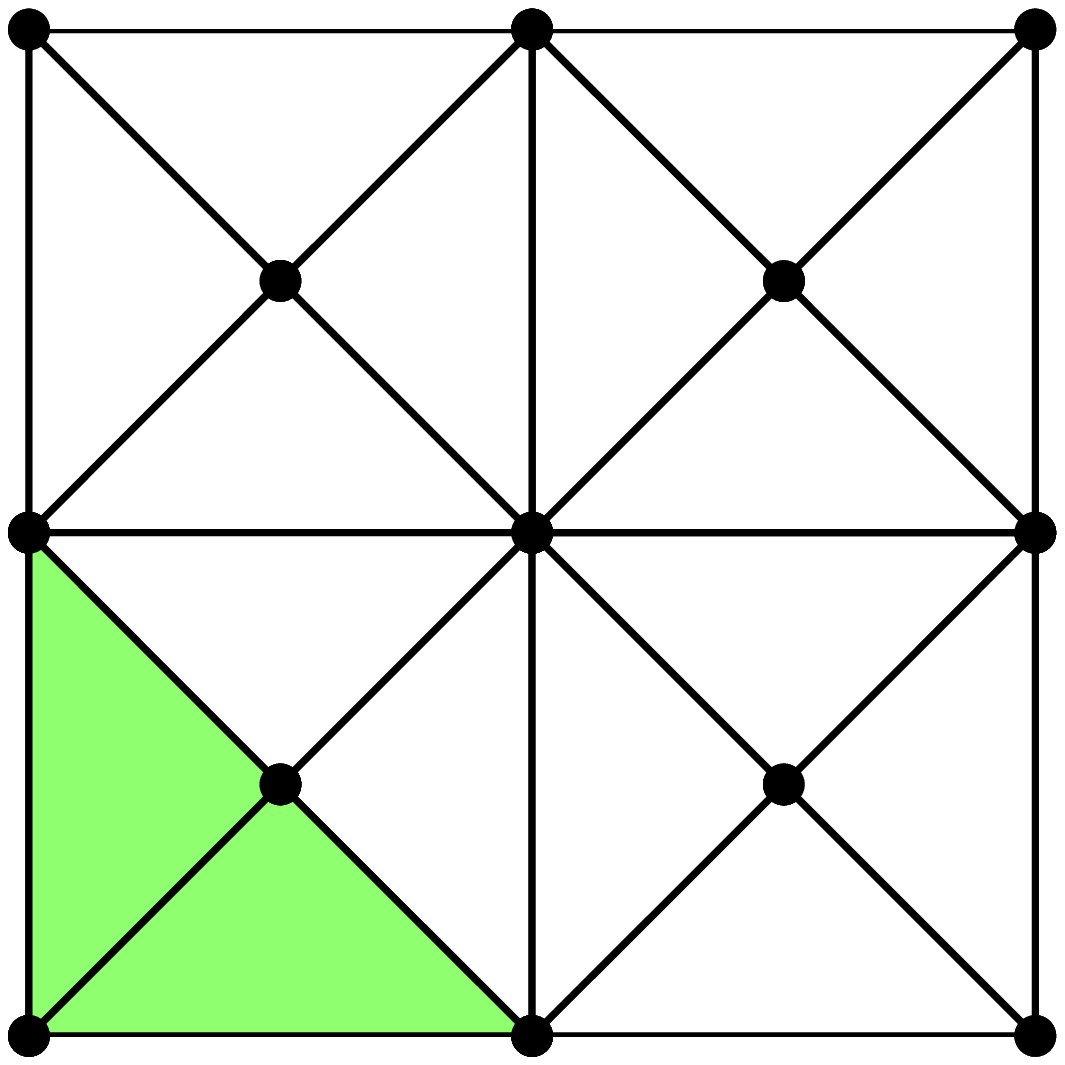} \quad
 \includegraphics[width=65mm]{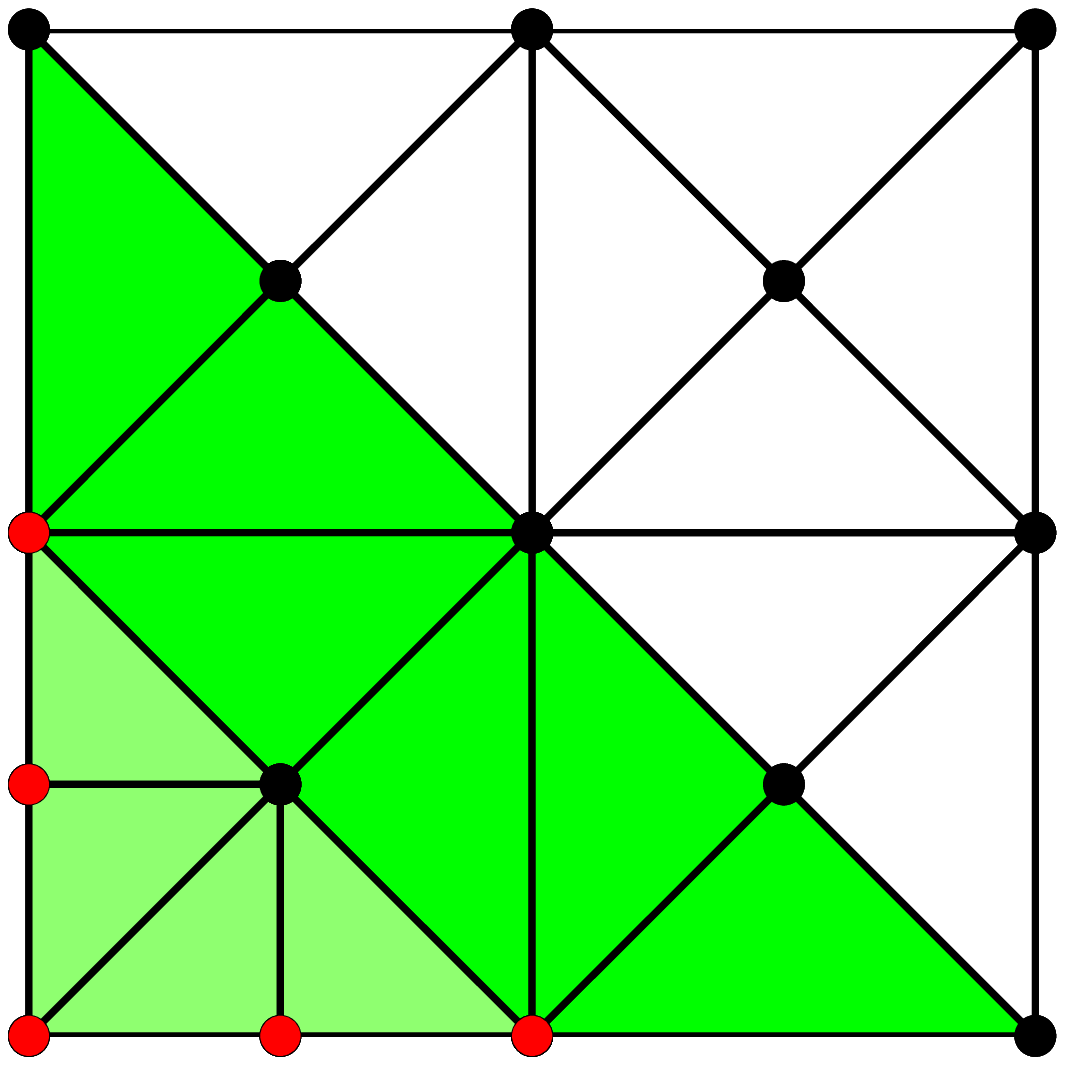} 
 \caption{
 The left figure shows a FEM mesh $\TT_{\ell-1}^\Omega$, where the two elements
 (green) are marked for refinement. Bisection of these two elements provides the mesh
 $\TT_\ell^\Omega$ (right), where two \textit{new nodes} are created.
 The set $\widetilde\NN_\ell^\Omega$ consists of these new nodes plus their 
 \textit{immediate neighbours} (red), where the corresponding patches have changed.
 The union of the support of the basis functions associated to nodes in
 $\widetilde\NN_\ell^\Omega$ is given by the light- and dark-green areas in the right figure.
 }
 \label{fig:Ntilde}
\end{figure}

\subsection{Galerkin system and block-diagonal preconditioning}
Let $\{\eta_{z_j}^\ell\}_{j=1}^{N_\ell}$ with $N_\ell :=
\# \NN_\ell^\Omega$ denote the nodal basis of $\XX^\ell$,
and let $\{\psi_{T_j}\}_{j=1}^{M_\ell}$ with $M_\ell := \# \TT_\ell^\Gamma$ denote a basis of $\YY^\ell$, where
$\psi_{T_j}$ denotes the characteristic function on $T_j \in\TT_\ell^\Gamma$.
The Galerkin matrix $\AA_\jn = \AA_\jn^\ell \in \R^{(N_\ell+M_\ell)\times
(N_\ell+M_\ell)}$ of the operator $\jn$ has the form
\begin{align*}
  \AA_\jn = \begin{pmatrix}
    \AA_\material & - \Mmat^T  \\
    \tfrac12 \Mmat-\Kmat & \AA_\slp
  \end{pmatrix}
  + \stabmat\stabmat^T,
\end{align*}
where the block matrices $\AA_\material\in\R^{N_\ell\times N_\ell}$,
$\AA_\slp\in\R^{M_\ell\times M_\ell}$, $\Mmat \in\R^{M_\ell\times N_\ell}$, and
$\Kmat\in\R^{M_\ell\times N_\ell}$ as well as the stabilization (column) vector
$\stabmat \in\R^{N_\ell+M_\ell}$ are given by
\begin{subequations}\label{eq:galerkin:matrices}
\begin{align}
  (\AA_\material)_{jk} &= \dual{\material\nabla\eta_{z_k}^\ell}{\nabla\eta_{z_j}^\ell}_\Omega
  \quad\text{for }j,k=1,\dots,N_\ell, \\
  (\AA_\slp)_{jk} &= \dual{\psi_{T_j}}{\slp\psi_{T_k}}_\Gamma \quad\text{for }
  j,k=1,\dots,M_\ell, \\
  (\Mmat)_{jk} &= \dual{\psi_{T_j}}{(\eta_{z_k}^\ell)|_\Gamma}_\Gamma
  \quad\text{for } j=1,\dots,M_\ell, k=1,\dots,N_\ell, \\
  (\Kmat)_{jk} &= \dual{\psi_{T_j}}{\dlp (\eta_{z_k}^\ell)|_\Gamma}_\Gamma
  \quad\text{for } j=1,\dots,M_\ell, k=1,\dots,N_\ell, \\
\begin{split}
  (\stabmat)_j &= \dual{1}{(\tfrac12-\dlp)(\eta_{z_j}^\ell)|_\Gamma}_\Gamma
  \quad\text{for }j=1,\dots,N_\ell, \\
  (\stabmat)_{j+N_\ell} &= \dual{1}{\slp\psi_{T_j}}_\Gamma
  \quad\text{for }j=1,\dots,M_\ell.
\end{split}
\end{align}
\end{subequations}
We stress that $\AA_\material$ as well as $\Mmat$ are sparse, whereas $\AA_\slp$
is dense. Note that the number of non-zeros in the matrix $\Kmat$ is bounded by
$\#(\NN_\ell^\Omega \cap \Gamma) \cdot M_\ell$.
Moreover, the application of the rank-$1$ stabilization matrix
$\stabmat\stabmat^T$ can be implemented efficiently with complexity
$\OO(N_\ell+M_\ell)$ for use with an iterative solver.

The discrete variational formulation~\eqref{eq:varform:stabilized} is equivalent to solving
the following linear system of equations: Find $\UU \in \R^{N_\ell+M_\ell}$ such that
\begin{align}\label{eq:jn_mat}
  \AA_\jn \UU  = \FF \in\R^{N_\ell+M_\ell},
\end{align}
where the right-hand side vector $\FF\in\R^{N_\ell+M_\ell}$ reads
\begin{align*}
  (\FF)_j &= 
    \dual{F}{(\eta_{z_j}^\ell,0)} \quad\text{for }j=1,\dots,N_\ell, \\
  (\FF)_{j+N_\ell} &= \dual{F}{(0,\psi_{T_j})} \quad\text{for }j=1,\dots,M_\ell.
\end{align*}
To formulate our block-diagonal preconditioner, we require an appropriate
operator $\fem$ which is related to the FEM-domain part of $\jn$.
The next lemma follows from a Rellich compactness argument,
since $\dlp c = -c/2$ for all constants $c\in\R$.
Details are analogous to, e.g.,~\cite[Lemma~10]{affkmp} and therefore
left to the reader.

\begin{lemma}\label{lemma:A}
  For $u,v\in H^1(\Omega)$, define
  \begin{subequations}\label{eq:femoperator}
  \begin{align}
    \dual{\fem u}v := \dual{\material \nabla u}{\nabla v}_\Omega +
    \dual{1}{(\tfrac12-\dlp)u}_\Gamma\dual{1}{(\tfrac12-\dlp)v}_\Gamma.
  \end{align}
  Then, the operator $\fem: H^1(\Omega)\to (H^1(\Omega))^*$ is linear, symmetric, continuous, and elliptic, and the constants
  \begin{align}\label{eq:const:A}
    c_\fem := \inf_{0\neq u\in H^1(\Omega)}\frac{\dual{\fem
    u}{u}}{\norm{u}{H^1(\Omega)}^2}
    \quad\text{and}\quad  
    C_\fem := \sup_{0\neq u\in H^1(\Omega)}\frac{\norm{\fem u}{(H^1(\Omega))^*}}{\norm{u}{H^1(\Omega)}}
  \end{align}
  \end{subequations}
  satisfy $0<c_\fem \le C_\fem < \infty$ and depend only on $c_\material$ and
  $C_\material$ from~\eqref{eq:const:material} as well as on $\Omega$.\qed
\end{lemma}

In this work, we investigate block-diagonal preconditioners of the form
\begin{align}\label{eq:precond_mat}
  \PPrec_\jn = \begin{pmatrix}
    \PPrec_\fem &  \bignull\\
     \bignull & \PPrec_\slp
  \end{pmatrix},
\end{align}
where $\PPrec_\fem$ is a ``good'' approximation of the Galerkin
matrix $\AA_\fem$ corresponding to the operator $\fem$ from Lemma~\ref{lemma:A} with respect to
the nodal basis of $\XX^\ell$, and
$\PPrec_\slp$ is a ``good'' approximation of the Galerkin matrix $\AA_\slp$.
Our construction below ensures that $\PPrec_\fem, \PPrec_\slp$ and hence $\PPrec$ are symmetric and positive
definite.

Instead of~\eqref{eq:jn_mat}, we solve the preconditioned system
\begin{align}\label{eq:jn_prec}
  \PPrec_\jn^{-1} \AA_\jn \UU = \PPrec_\jn^{-1}\FF \in\R^{N_\ell+M_\ell}.
\end{align}
For this non-symmetric system of linear equations, we use a
preconditioned GMRES algorithm~\cite{saad}, which will be discussed in the
following subsection.
The preconditioned Galerkin matrix reads in block form 
\begin{align}
  \PPrec_\jn^{-1} \AA_\jn = \begin{pmatrix}
    \PPrec_\fem^{-1} & \bignull \\
            \bignull & \PPrec_\slp^{-1}
  \end{pmatrix}
  \left[ \begin{pmatrix}
    \AA_\material & - \Mmat^T  \\
    \tfrac12 \Mmat-\Kmat & \AA_\slp
  \end{pmatrix}
+ \stabmat\stabmat^T \right] .
\end{align}

\subsection{Preconditioned GMRES algorithm}\label{sec:gmres}
Let $\PPrec\in\R^{N\times N}$ denote a symmetric and positive definite matrix
and let $\AA \in \R^{N\times N}$
denote a (possibly) non-symmetric, but positive definite matrix.
Let, $\ee^k\in\R^N$ denote the standard unit vector with entries $(\ee^k)_j = \delta_{kj}$.
The preconditioned GMRES
algorithm reads as follows.
\begin{algorithm}[GMRES]\label{alg:gmres}
  {\bf Input: } Matrices $\PPrec,\AA \in\R^{N\times N}$, right-hand side vector
  $\FF\in\R^N$, initial guess $\UU^0\in\R^N$, relative tolerance $\tau>0$, and maximum number of iterations
  $K\in\N$ with $K\leq N$.
  \begin{itemize}
    \item[(a)] Allocate memory for the matrix $\mathbf H
      \in\R^{(K+1)\times K}$, the vectors $\VV^i\in\R^N$, $i=1,\dots,K$, $\WW\in\R^N$,
      $\UU\in\R^N$, $\RR^0\in\R^N$ and $\RR\in\R^N$.
    \item[(b)] Compute initial residual $\RR^0 \leftarrow \PPrec^{-1}
      (\FF-\AA\UU^0)$ and $\VV^1 \leftarrow \RR^0/\norm{\RR^0}\PPrec$.
    \item[(c)] Set counter $k\leftarrow 1$, and initialize $(\mathbf H)_{ij}
      \leftarrow 0$ for all $i=1,\dots,K+1, j=1,\dots, K$.
  \end{itemize}
  Iterate the following steps \emph{(i)--(vii)}:
  \begin{itemize}
    \item[(i)] Compute $\WW \leftarrow \PPrec^{-1} \AA\VV^k$.
    \item[(ii)] For all $i=1,\dots,k$ compute
      \begin{align*}
        (\mathbf{H})_{ik} \leftarrow \dual{\WW}{\VV^i}_\PPrec
        \quad\text{and}\quad \WW \leftarrow \WW - \VV^i (\mathbf{H})_{ik}.
      \end{align*}
    \item[(iii)] Compute $(\mathbf{H})_{k+1,k} \leftarrow \norm\WW\PPrec$.
    \item[(iv)] Define the sub-matrix $\overline{\mathbf H}^k
      \in\R^{(k+1) \times k}$
      with entries $(\overline{\mathbf{H}}^k)_{ij} := (\mathbf{H})_{ij}$ for
      $i=1,\dots,k+1, j=1,\dots,k$ and compute
      \begin{align*}
        \yy^k \leftarrow \argmin\limits_{\yy\in\R^k} \norm{\norm{\RR_0}\PPrec
        \ee^1 - \overline{\mathbf{H}}^k \yy}2.
      \end{align*}
    \item[(v)] Compute $\UU \leftarrow \UU^0 + (\VV^1\cdots\VV^k)\yy^k$ and
      $\RR \leftarrow \PPrec^{-1} (\FF-\AA\UU)$.
    \item[(vi)] Stop iteration if $\norm{\RR}\PPrec \leq \tau
      \norm{\RR^0}\PPrec$ or $k\geq
      K$.
    \item[(vii)] Otherwise, compute $\VV^{k+1} \leftarrow
      \WW/(\mathbf{H})_{k+1,k}$, update counter $k\leftarrow k+1$, and goto
      (iv).
  \end{itemize}
  {\bf Output: } $\UU$, $k$, and $\norm{\RR}\PPrec/\norm{\RR^0}\PPrec$.
\end{algorithm}
\noindent
Note that for $\PPrec$ being the identity matrix, Algorithm~\ref{alg:gmres}
is the usual GMRES algorithm with inner product $\dual\cdot\cdot_2$, see
e.g.~\cite{saad}.
The main memory consumption is given by the vectors
$\VV^i\in\R^N$, while the matrix
$\overline{\mathbf H}^k\in \R^{(k+1)\times k}$ in step (iv) is a sub-block of the matrix $\mathbf H
\in\R^{(K+1)\times K}$ and thus does not need to be stored explicitly.

As is often the case for multilevel preconditioners, the application of
$\PPrec^{-1}$ on a vector is known, whereas the application of $\PPrec$ is
unknown. We therefore note that the GMRES
Algorithm~\ref{alg:gmres} can be implemented without using $\PPrec$ to compute
the inner products $\dual\cdot\cdot_\PPrec$ and the norms $\norm\cdot\PPrec$.
To this end, one replaces the computation of $\VV^1$ by $\widetilde\VV^1 \leftarrow
(\FF-\AA\UU^0)/\norm{\RR^0}\PPrec$, where $\norm{\RR^0}\PPrec^2 =
\dual{\PPrec^{-1}(\FF-\AA\UU^0)}{\FF-\AA\UU^0}_2$. Then, $\VV^1 \leftarrow
\PPrec^{-1} \widetilde\VV^1$.
In step (i), we compute $\widetilde\WW \leftarrow \AA\VV^k$ instead of $\WW$.
Step (ii) is replaced by
$(\mathbf H)_{ik} \leftarrow \dual{\PPrec^{-1}\widetilde\WW}{\VV^i}_\PPrec =
\dual{\widetilde\WW}{\VV^i}_2$ and $\widetilde\WW\leftarrow
\widetilde\WW-\widetilde\VV^i (\mathbf H)_{ik}$.
Instead of step (iii), we then compute $(\mathbf H)_{k+1,k} \leftarrow
\norm{\PPrec^{-1}\widetilde\WW}\PPrec =
\sqrt{\dual{\PPrec^{-1}\widetilde\WW}{\widetilde\WW}_2}$.
Note that $\norm{\RR}\PPrec^2 = \dual{\PPrec^{-1}(\FF-\AA\UU)}{\FF-\AA\UU}_2$ in
step (vi). Lastly, in step (vii) we replace the computation of $\VV^{k+1}$ by
$\widetilde\VV^{k+1} \leftarrow \widetilde\WW / (\mathbf H)_{k+1,k}$ and
$\VV^{k+1} \leftarrow
\PPrec^{-1}\widetilde\VV^{k+1}$.

\subsection{Local multilevel preconditioner and main
result}\label{sec:main:LMLD}
For both the FEM part $\PPrec_\fem$ and BEM part
$\PPrec_\slp$ in~\eqref{eq:precond_mat}, we will use local multilevel preconditioners
which are optimal in the sense that the condition numbers of the preconditioned
systems are independent of the number of levels $L$ and the mesh-size $h_L$.

For the preconditioner $\PPrec_\fem$, we use an additive Schwarz multilevel diagonal
preconditioner similar to the one in~\cite{xch10}. 
Recall $\widetilde\NN_\ell^\Omega \subseteq\NN_\ell^\Omega$
from~\eqref{def:Ntilde}
Define
\begin{align}
  \widetilde\XX^\ell := \linhull\set{\eta_z^\ell}{z
  \in\widetilde\NN_\ell^\Omega} \subseteq\XX^\ell.
\end{align}
Let $\widetilde\embedding^\ell : \widetilde\XX^\ell \to \XX^L$ denote the canonical
embedding with matrix representation $\widetilde\embmat^\ell \in \R^{\widetilde
N_\ell^\Omega \times N_L^\Omega}$ and $\widetilde N_\ell^\Omega := \#
\widetilde\NN_\ell^\Omega$. 
Furthermore, let $\widetilde\DD_\fem^\ell$ denote the diagonal of the Galerkin matrix
$\widetilde\AA_\fem^\ell$ with respect to the local set of nodes
$\widetilde\NN_\ell^\Omega$, i.e.
$(\widetilde\AA_\fem^\ell)_{jk} = \dual{\fem \eta_{z_k}^\ell}{\eta_{z_j}^\ell}$
for $j,k=1,\dots,\widetilde N_\ell^\Omega$ and $(\widetilde\DD_\fem^\ell)_{jk} :=
\delta_{jk} (\widetilde\AA_\fem^\ell)_{jj}$.
Then, our local multilevel diagonal preconditioner $\PPrec_\fem$ is defined via
\begin{align}\label{eq:defPrecFEM}
  \PPrec_\fem^{-1} := (\PPrec_\fem^L)^{-1} := \sum_{\ell=0}^L
  \widetilde\embmat^\ell (\widetilde\DD_\fem^\ell)^{-1} (\widetilde\embmat^\ell)^T.
\end{align}
From the definition, we see that this preconditioner corresponds to a
diagonal scaling on each level, where scaling is done on the local subset
$\widetilde\NN_\ell^\Omega$ only.

For all boundary nodes $z\in\NN_\ell^\Gamma$, let
\begin{align}
  \zeta_z^\ell \in \ZZ^\ell := \set{v\in C(\Gamma)}{v|_T \text{ is affine for all }T\in\TT_\ell^\Gamma}
\end{align}
denote the
boundary hat-function with $\zeta_z^\ell(z') = \delta_{zz'}$ for all
$z'\in\NN_\ell^\Gamma$.
To construct an efficient preconditioner $\PPrec_\slp$ for the weakly-singular
integral operator $\slp$ in 2D, we use the Haar-basis functions
$\chi_z^\ell := (\zeta_z^\ell)'$ for all $z\in\NN_\ell^\Gamma$.
Recall $\widetilde\NN_\ell^\Gamma\subseteq \NN_\ell^\Gamma$
from~\eqref{def:Ntilde}.
Let  $\widetilde N_\ell^\Gamma := \# \widetilde\NN_\ell^\Gamma$.
Define the local subspaces
\begin{align}
  \widetilde\YY^\ell := \linhull\set{\chi_z^\ell}{z\in\widetilde\NN_\ell^\Gamma} \subsetneqq
  \YY^\ell, \qquad \widetilde\ZZ^\ell :=
  \linhull\set{\zeta_z^\ell}{z\in\widetilde\NN_\ell^\Gamma} \subseteq\ZZ^\ell, 
\end{align}
and the matrix $\widetilde\Haar^\ell \in\R^{M_\ell \times \widetilde
N_\ell^\Gamma}$
which represents the Haar-basis functions with respect
to the canonical basis $\set{\psi_T}{\psi_T \text{ is characteristic function on }
T\in\TT_\ell^\Gamma}$ of $\YY^\ell$, i.e.
\begin{align*}
  \chi_{z_k}^\ell = \sum_{j=1}^{M_\ell} (\widetilde\Haar^\ell)_{jk} \psi_{T_j}
  \quad\text{for } j=1,\dots,M_\ell, 
  \quad k=1,\dots,\widetilde N_\ell^\Gamma.
\end{align*}
Maue's formula~\cite{maue} states the relation
\begin{align}\label{eq:maueformula}
  \dual{\hyp v}w_\Gamma = \dual{\slp v'}{w'}_\Gamma \quad\text{for all } v,w\in
  H^1(\Gamma)
\end{align}
and thus reveals the identity
\begin{align*}
  \Big( (\widetilde\Haar^\ell)^T \AA_\slp^\ell \widetilde\Haar^\ell \Big)_{jk} =
  \dual{\slp \chi_{z_k}^\ell}{\chi_{z_j}^\ell}_\Gamma = 
  \dual{\hyp \zeta_{z_k}^\ell}{\zeta_{z_j}^\ell}_\Gamma =:
  (\widetilde\AA_\hyp^\ell)_{jk} \quad\text{for }j,k =
  1,\dots,\widetilde N_\ell^\Gamma
\end{align*}
for the Galerkin matrix $\widetilde \AA_\hyp^\ell$ of the hypersingular integral
operator with respect to the nodes $\widetilde\NN_\ell^\Gamma$.
Let $\widetilde\DD_\hyp^\ell$ denote the diagonal of $\widetilde\AA_\hyp^\ell$,
and let $\embedP^\ell : \YY^\ell \to \YY^L$ denote the canonical embedding with
matrix representation $\embedPmat^\ell$.
Moreover, define $D := \dual{1}{\slp 1}_\Gamma$ and let $\onemat \in\R^{M_L}$
denote the vector with constant entries $(\onemat)_j = 1$ for all $j=1,\dots M_L$.
Then, our multilevel diagonal preconditioner $\PPrec_\slp$ for the
weakly-singular integral operator reads
\begin{align}\label{eq:defPrecSLP}
  \PPrec_\slp^{-1} := (\PPrec_\slp^L)^{-1} := \onemat D^{-1} \onemat^T + \sum_{\ell=0}^L 
  \embedPmat^\ell \widetilde\Haar^\ell (\widetilde\DD_\hyp^\ell)^{-1}
  (\widetilde\Haar^\ell)^T (\embedPmat^\ell)^T.
\end{align}

The following theorem is the main result of this work. 
Let $\cond_{\mathbf C}(\AA) = \norm{\AA}{\mathbf C}
\norm{\AA^{-1}}{\mathbf C}$ denote the condition number of the matrix $\AA$ with
respect to the norm $\norm\cdot{\mathbf C}$ induced by the symmetric and
positive definite matrix $\mathbf C$.

\begin{theorem}\label{thm:main}
Let $\PPrec_\fem$ resp. $\PPrec_\slp$ denote the multilevel preconditioners
defined in~\eqref{eq:defPrecFEM} resp.~\eqref{eq:defPrecSLP}.
Then, the condition number
\begin{align}\label{eq:asdf}
\cond_{\PPrec_\jn}(\PPrec_\jn^{-1}\AA_\jn) \leq C < \infty
\end{align}
is uniformly bounded. Moreover,
the $j$-th residual $\RR^j$ from the preconditioned GMRES Algorithm~\ref{alg:gmres}
with $\PPrec = \PPrec_\jn$ from~\eqref{eq:precond_mat} satisfies
\begin{align}
  \norm{\RR^j}{\PPrec_\jn} \leq q_{\rm GMRES}^j \norm{\RR^0}{\PPrec_\jn}.
\end{align}
The constants $C>0$ and $0<q_{\rm GMRES}<1$ depend only on $\Omega$, the
ellipticity and continuity constants of the material tensor $A$
from~\eqref{eq:const:material}, the initial triangulations $\TT_0^\Omega$ and
$\TT_0^\Gamma$, as well as on the mesh-refinement strategy chosen.
\end{theorem}


\section{Spectral estimates for $\AA_\fem$ and $\AA_\slp$}
\label{sec:spectral:estimates}

\noindent
In this section, we provide spectral estimates for the matrices $\AA_\fem$,
$\AA_\slp$.
In particular, the equivalences $\dual{\AA_\fem\xx}\xx_2 \simeq \dual{\PPrec_\fem
\xx}\xx_2$ for $\xx\in\R^{N_L}$ as well as $\dual{\AA_\slp\pphi}\pphi_2 \simeq
\dual{\PPrec_\slp \pphi}\pphi_2$ for $\pphi\in\R^{M_L}$ are optimal in the sense, that the
involved constants are independent of $L$ and $h_L$.

The remainder of this section is organized as follows:
In Section~\ref{sec:spectral:estimates:fem}, we focus on the optimality of
the preconditioner $\PPrec_\fem$, which follows from~\cite{wuchen06,xch10}.
In Section~\ref{sec:spectral:estimates:slp}, we analyze the preconditioner
$\PPrec_\slp$ and prove optimality thereof. Note that optimality of
$\PPrec_\slp$ for uniform meshes has already been proved in~\cite{transtep96},
where $\widetilde\NN_\ell^\Omega = \NN_\ell^\Omega$ and
$\widetilde\NN_\ell^\Gamma = \NN_\ell^\Gamma$,
while optimality on adaptive meshes is a particular contribution of
the present work.

\subsection{Optimality of the multilevel preconditioner $\PPrec_\fem$}
\label{sec:spectral:estimates:fem}
Define the local projection operators $\precfem_z^\ell : \XX^L\to \XX_z^\ell :=
\linhull\{\eta_z^\ell\}$ by 
\begin{align}
  \dual{\precfem_z^\ell v}{w_z^\ell}_\fem = \dual{v}{w_z^\ell}_{\fem} \quad\text{for all }
  w_z^\ell \in\XX_z^\ell,
\end{align}
and the multilevel additive Schwarz preconditioner
\begin{align}
  \PASfem := \sum_{\ell=0}^L \sum_{z\in\widetilde\NN_\ell} \precfem_z^\ell \,:\, \XX^L
  \to \XX^L.
\end{align}
A straightforward calculation shows the identity
\begin{align}
  \dual{\PPrec_\fem^{-1} \AA_\fem \xx}\xx_{\AA_\fem} = \dual{\PASfem v}v_\fem,
\end{align}
where $v\in\HH^L$ is given by $v = \sum_{j=1}^{N_L} (\xx)_j \eta_{z_j}^L$,
and $N_L := \# \NN_L^\Omega$ denotes the number of nodes in the FEM
domain.
Thus, bounds for the extremal eigenvalues of the operator $\PASfem$ provide bounds
for the extremal eigenvalues of the preconditioned system.

\begin{theorem}\label{thm:precond:fem}
  The preconditioner matrix $\PPrec_\fem$ is symmetric and positive definite.
  There holds
  \begin{align}\label{eq:specest:fem}
    d_\fem \dual{\PPrec_\fem\xx}\xx_2 \leq \dual{\AA_\fem\xx}\xx_2
    \leq D_\fem \dual{\PPrec_\fem\xx}\xx_2 \quad\text{for all }\xx\in\R^{N_L}.
  \end{align}
  The constants $d_\fem,D_\fem$ depend only on $\Omega$, the initial
  triangulation $\TT_0^\Omega$, as well as the use of newest vertex bisection for
  mesh-refinement.
\end{theorem}

\begin{proof}
There holds a similar result to~\cite[Theorem~4.2]{xch10} for the additive
Schwarz operator $\PASfem$, i.e., $\PASfem$ is $\dual\cdot\cdot_\fem$-symmetric and
there holds for all $v\in \XX^L$
\begin{align}\label{eq:equivalence:fem}
  \dual{\PASfem v}v_\fem \simeq \dual{\fem v}v
\end{align}
where the hidden constants depend only on $\Omega$, the initial triangulation
$\TT_0^\Omega$, as well as on the mesh-refinement
chosen.
In particular, the proof of the equivalence~\eqref{eq:equivalence:fem} follows the lines
of~\cite[Section~4.1]{xch10}.
The two key ingredients of the proof are the results~\cite[Lemma~3.2--3.3]{wuchen06},
which also hold for the problem considered in this work.
\end{proof}

\subsection{Optimality of the multilevel preconditioner $\PPrec_\slp$}
\label{sec:spectral:estimates:slp}
In this section we prove that the optimal additive Schwarz preconditioner
for the hypersingular integral operator provided in~\cite{ffps}, which is based
on a space decomposition of lowest-order hat-functions, induces optimality of
the additive Schwarz operator for the weakly-singular integral operator.
The key ingredient of the proof is Maue's
formula~\eqref{eq:maueformula}, which allows
us, roughly speaking, to change between the $H^{1/2}$ and $H^{-1/2}$ norms. 
For uniform meshes, a similar approach, which uses a generalised antiderivative
operator~\cite{hs96}, is considered in~\cite{transtep96}. 
The remainder of this section can be seen as an alternate proof of the
results from~\cite[Section~3]{transtep96} as well as an extension to locally
refined meshes.

\begin{theorem}\label{thm:precond:slp}
  The preconditioner matrix $\PPrec_\slp$ is symmetric and positive definite.
  There holds
  \begin{align}\label{eq:specest:slp}
    d_\slp \dual{\PPrec_\slp\pphi}\pphi_2 \leq \dual{\AA_\slp\pphi}\pphi_2
    \leq D_\slp \dual{\PPrec_\slp\pphi}\pphi_2 \quad\text{for all }\pphi\in\R^{M_L}.
  \end{align}
  The constants $d_\slp,D_\slp$ depend only on $\Gamma$, the initial
  triangulation $\TT_0^\Gamma$, and the chosen mesh-refinement.
  Moreover, the eigenvalues of the preconditioned matrix
    $\PPrec_\slp^{-1} \AA_\slp$ are bounded by
    \begin{align}
      d_\slp \leq \evmin(\PPrec_\slp^{-1} \AA_\slp) \leq \evmax(\PPrec_\slp^{-1}
      \AA_\slp) \leq D_\slp.
    \end{align}
\end{theorem}

According to~\eqref{def:slp},
and~\eqref{eq:bio:symmetry}--\eqref{eq:slp:ellipticity}, $\dual{\phi}\psi_\slp
:= \dual{\slp\phi}\psi_\Gamma$ defines a scalar product with equivalent norm
$\norm\phi\slp^2 := \dual\phi\phi_\slp$ on $H^{-1/2}(\Gamma)$.
According
to~\eqref{def:hyp},~\eqref{eq:bio:symmetry},~\eqref{eq:hyp:semielliptic}, and
the Rellich compactness theorem, $\dual{u}v_\hyp := \dual{\hyp u}v_\Gamma +
\dual{u}1_\Gamma\dual{v}1_\Gamma$ defines a scalar product with equivalent norm
$\norm{u}\hyp^2 := \dual{u}u_\hyp$ on $H^{1/2}(\Gamma)$.

We need to define a subspace decomposition of the space $\YY^L$.
To that end, we make use of the Haar basis functions $\chi_z^\ell =
(\zeta_z^\ell)'$, which are the arclength
derivatives of the hat functions $\zeta_z^\ell$.
Then, $\ZZ^L = \linhull\set{\zeta_z^L}{z\in\NN^L}$ can be decomposed into
\begin{align*}
  \ZZ^L = \sum_{\ell=0}^L \sum_{z\in\widetilde\NN_\ell^\Gamma} \ZZ_z^\ell
  \quad\text{with } \ZZ_z^\ell := \linhull\{\zeta_z^\ell\}.
\end{align*}
Moreover, simple calculations with $\dual{\chi_z^\ell}1_\Gamma=0$ show that
\begin{align*}
  \YY^L = \YY^{00} \oplus \sum_{\ell=0}^L \sum_{z\in\widetilde\NN_\ell^\Gamma} \YY_z^\ell
  \quad\text{with } \YY^{00} := \linhull\{1\} \text{ and } \YY_z^\ell := \linhull\{\chi_z^\ell\}
\end{align*}
defines a (direct sum) decomposition of $\YY^L$ into $\YY^{00}$ and an additive
Schwarz space.
With this,
we define the additive Schwarz operator
\begin{align}\label{def:slp:as}
  \PASslp = \precslp^{00} + \sum_{\ell=0}^L \sum_{z\in\widetilde\NN_\ell^\Gamma} \precslp_z^\ell,
\end{align}
where $\precslp^{00}$ resp. $\precslp_z^\ell$ are defined for all $\phi\in\YY^L$ via
\begin{subequations}\label{def:slp:proj}
\begin{align}  
  \dual{\precslp^{00} \phi}{\phi_0}_\slp &= \dual\phi{\phi_0}_\slp
  \quad\text{for all } \phi_0 \in\YY^{00}
  \quad\text{resp.} \\
  \dual{\precslp_z^\ell \phi}{\phi_z^\ell}_\slp &= \dual\phi{\phi_z^\ell}_\slp \quad\text{for all } 
  \phi_z^\ell \in\YY_z^\ell.
\end{align}
\end{subequations}
We note that the symmetry of the orthogonal projectors $\precslp^{00}$
resp. $\precslp_z^\ell$ implies that also $\PASslp$ is symmetric
\begin{align}\label{eq:PASslp:sym}
  \dual{\PASslp \phi}\psi_\slp = \dual\phi{\PASslp\phi}_\slp \quad\text{for all } 
  \phi,\psi \in\YY^L.
\end{align}
Our analysis of $\PASslp$ builds on own results~\cite{ffps} on the additive
Schwarz operator associated to the hypersingular integral equation,
\begin{align}
  \PAShyp = \sum_{\ell=0}^L \sum_{z\in\widetilde\NN_\ell^\Gamma}
  \prechyp_z^\ell, \quad\text{where }   \dual{\prechyp_z^\ell v}{v_z^\ell}_\hyp = \dual{v}{v_z^\ell}_\hyp \quad\text{for
  all } v_z^\ell \in\XX_z^\ell.
\end{align}
The analysis of~\cite{ffps} provides the following result.
\begin{lemma}[{\cite[Proposition~4 ]{ffps}}]\label{lem:hypAS}
  The operator $\PAShyp : \ZZ^L\to\ZZ^L$ is symmetric and satisfies
  \begin{align*}
    \c{hyp:lb} \norm{v}\hyp^2 \leq \dual{\PAShyp v}v_\hyp \leq \c{hyp:ub}
    \norm{v}\hyp^2 \quad\text{for all } v\in\ZZ^L,
  \end{align*}
  where the constants $\setc{hyp:lb},\setc{hyp:ub}>0$ depend only on $\Gamma$,
  the initial triangulation $\TT_0^\Gamma$, as well as on the chosen
  mesh-refinement.
  \qed
\end{lemma}

For each $\phi\in\YY^L$, we follow~\cite{transtep96} and split
\begin{align}\label{eq:splitPhi}
  \phi = \phi_0 + \widetilde\phi \quad\text{with unique } \phi_0 =
  \dual{\phi}1_\Gamma/|\Gamma| \in\YY^{00}\,
  \text{and } 
  \widetilde\phi\in \sum_{\ell=0}^L \sum_{z\in\widetilde\NN_\ell^\Gamma} \YY_z^\ell.
\end{align}
Let us introduce a mechanism to switch between the $H^{1/2}$ and $H^{-1/2}$
norms. For $\widetilde\phi\in \YY_*^L := \{ \phi \in\YY^L \,:\,\dual{\phi}1_\Gamma = 0\}$,
there exists a unique function $\widetilde v \in\ZZ_*^L := \{ v\in\ZZ^L \,:\, \dual{v}1_\Gamma = 0 \}$ such that 
\begin{align}
  \widetilde\phi = (\widetilde v)'.
\end{align}
To see this, let $\widetilde\phi \in \YY_*^L$ with
\begin{align}
  \widetilde\phi = \sum_{\ell=0}^L \sum_{z\in\widetilde\NN_\ell^\Gamma} \alpha_z^\ell
  \chi_z^\ell \quad\text{and define}\quad v:= \sum_{\ell=0}^L
  \sum_{z\in\widetilde\NN_\ell^\Gamma} \alpha_z^\ell \zeta_z^\ell \in\ZZ^L.
\end{align}
Then, $\widetilde v := v- \dual{v}1_\Gamma/|\Gamma| \in \ZZ_*^L$, and
it holds $(\widetilde v)' = v' = \widetilde\phi$ as $\chi_z^\ell =
(\zeta_z^\ell)'$.
Maue's formula~\eqref{eq:maueformula} provides the important
identities 
\begin{align}\label{eq:nedelec}
  \dual{\hyp\zeta_z^\ell}{\zeta_z^\ell}_\Gamma =
  \dual{\slp\chi_z^\ell}{\chi_z^\ell}_\Gamma \quad\text{as well as}\quad
  \dual{\hyp\widetilde v}{\widetilde v}_\Gamma =
  \dual{\slp\widetilde\phi}{\widetilde\phi}_\Gamma.
\end{align}
We stress that~\eqref{eq:nedelec} allows to switch between the spaces
$H^{1/2}$ and $H^{-1/2}$. This is the heart of the proof of the following
proposition.
\begin{proposition}\label{thm:slpLMLD}
The operator $\PASslp : \YY^L\to\YY^L$ is
symmetric~\eqref{eq:PASslp:sym}, and it holds
\begin{align}\label{eq:slpLMLD}
  \c{slp_lb} \norm\psi\slp^2 \leq \dual{\PASslp \psi}\psi_\slp \leq\c{slp_ub}
  \norm\psi\slp^2 \quad\text{for all }\psi \in \YY^L.
\end{align}
The constants $\setc{slp_lb},\setc{slp_ub}>0$ depend only on $\Gamma$, the
initial triangulation $\TT_0^\Gamma$, as well as on the bisection algorithm from
Section~\ref{sec:refinement}.
\end{proposition}

For the proof of Proposition~\ref{thm:slpLMLD} we need the following result, see
e.g.~\cite{zhang92}, where the first part is known as Lions' lemma.

\begin{lemma}\label{lem:stabledec}
  (i)\, Let $c>0$ and $\psi\in\YY^L$. Suppose that there exists a decomposition $\psi
  = \psi_0 + \sum_{\ell=0}^L \sum_{z\in\widetilde\NN_\ell^\Gamma} \psi_z^\ell$ with $\psi_0
  \in\YY^{00}, \psi_z^\ell\in\YY_z^\ell$ such that
  \begin{align}
    \norm{\psi_0}\slp^2 + \sum_{\ell=0}^L\sum_{z\in\widetilde\NN_\ell^\Gamma}
    \norm{\psi_z^\ell}\slp^2 \leq c^{-1} \norm\psi\slp^2.
  \end{align}
  Then it follows, $c \norm\psi\slp^2 \leq \dual{\PASslp\psi}\psi_\slp$.

  (ii)\, Let $C>0$ and $\psi\in\YY^L$. Suppose that for all decompositions $\psi =
  \psi_0 + \sum_{\ell=0}^L \sum_{z\in\widetilde\NN_\ell^\Gamma} \psi_z^\ell$ with $\psi_0
  \in\YY^{00}$ and $\psi_z^\ell\in\YY_z^\ell$ holds
  \begin{align}
    \norm\psi\slp^2 \leq C \big( \norm{\psi_0}\slp^2 + \sum_{\ell=0}^L
    \sum_{z\in\widetilde\NN_\ell^\Gamma}  \norm{\psi_z^\ell}\slp^2 \big).
  \end{align}
  Then, it follows $\dual{\PASslp\psi}\psi_\slp \leq C \norm{\psi}\slp^2$.
\end{lemma}

\begin{proof}[Proof of Proposition~\ref{thm:slpLMLD}, lower bound in~\eqref{eq:slpLMLD}]
By means of Lemma~\ref{lem:stabledec}, we have to provide a stable subspace
decomposition. For $\phi\in\YY^L$, we consider the unique decomposition
$\phi = \phi_0 + \widetilde\phi$ from~\eqref{eq:splitPhi}.
With $\phi_0 = \dual{\phi}1_\Gamma/|\Gamma|$, we infer
\begin{align}\label{eq:slpLMLD2}
  \norm{\phi_0}\slp = \frac{\dual{\phi}1_\Gamma}{|\Gamma|}
  \norm{1}\slp \leq \frac{\norm{\phi}{H^{-1/2}(\Gamma)}
  \norm{1}{H^{1/2}(\Gamma)}}{|\Gamma|}
  \norm{1}{\slp} \lesssim \norm\phi\slp.
\end{align}
Moreover, there
exists $\widetilde v \in\ZZ_*^L$, such that $(\widetilde v)' =
\widetilde\phi$.
The abstract result~\cite[Lemma~3.1]{zhang92} states
\begin{align}\label{eq:proj:evmin}
  \evmin(\PAShyp) = \min\limits_{v\in\ZZ^L} 
  \frac{\norm{v}\hyp^2}{\min\limits_{\sum_{\ell=0}^L\sum_{z\in\widetilde\NN_\ell^\Gamma} v_z^\ell =
  v} \sum_{\ell=0}^L \sum_{z\in\widetilde\NN_\ell^\Gamma}
  \norm{v_z^\ell}\hyp^2},
\end{align}
since $\PAShyp$ is a finite sum of symmetric projectors. 
Lemma~\ref{lem:hypAS} provides uniform boundedness of the Rayleigh quotient
$\dual{\PAShyp v}v_\hyp / \norm{v}\hyp^2 \geq \c{hyp:lb}> 0$. Thus,
$\evmin(\PAShyp)\geq \c{hyp:lb}>0$ is uniformly
bounded, and we infer from~\eqref{eq:proj:evmin} the existence of a decomposition
$\widetilde v = \sum_{\ell=0}^L \sum_{z\in\widetilde\NN_\ell^\Gamma} v_z^\ell$
with $v_z^\ell \in\ZZ_z^\ell$ such that
\begin{align}\label{eq:slpLMLD3}
  \sum_{\ell=0}^L \sum_{z\in\widetilde\NN_\ell^\Gamma} \norm{v_z^\ell}\hyp^2
  \lesssim \norm{\widetilde v}\hyp^2.
\end{align}
This provides a decomposition of $\widetilde\phi$ into functions $\phi_z^\ell =
(v_z^\ell)'\in\YY_z^\ell$ by
\begin{align}\label{eq:slpLMLD4}
  \widetilde\phi = (\widetilde v)' =: \sum_{\ell=0}^L \sum_{z\in\widetilde\NN_\ell^\Gamma} \phi_z^\ell.
\end{align}
The identities from~\eqref{eq:nedelec} imply
\begin{align}\label{eq:slpLMLD5}
  \sum_{\ell=0}^L \sum_{z\in\widetilde\NN_\ell^\Gamma} \norm{\phi_z^\ell}\slp^2
   = \sum_{\ell=0}^L \sum_{z\in\widetilde\NN_\ell^\Gamma}
   \dual{\slp\phi_z^\ell}{\phi_z^\ell}_\Gamma
   = \sum_{\ell=0}^L \sum_{z\in\widetilde\NN_\ell^\Gamma}
   \dual{\hyp v_z^\ell}{v_z^\ell}_\Gamma
\end{align}
The estimate $\dual{\hyp v_z^\ell}{v_z^\ell}_\Gamma \leq \norm{v_z^\ell}\hyp^2$
and~\eqref{eq:slpLMLD3} prove
\begin{align}\label{eq:slpLMLD6}
  \sum_{\ell=0}^L \sum_{z\in\widetilde\NN_\ell^\Gamma} 
  \dual{\hyp v_z^\ell}{v_z^\ell}_\Gamma
  \leq \sum_{\ell=0}^L \sum_{z\in\widetilde\NN_\ell^\Gamma} 
  \norm{v_z^\ell}\hyp^2
  \lesssim \norm{\widetilde v}\hyp^2.
\end{align}
Since $\dual{\widetilde v}1_\Gamma = 0$, there holds
$\norm{\widetilde v}\hyp^2 = \dual{\hyp\widetilde v}{\widetilde v}_\Gamma$.
This,~\eqref{eq:slpLMLD5}--\eqref{eq:slpLMLD6}, and Maue's formula~\eqref{eq:nedelec}
yield
\begin{align}\label{eq:slpLMLD7}
  \sum_{\ell=0}^L \sum_{z\in\widetilde\NN_\ell^\Gamma} \norm{\phi_z^\ell}\slp^2
  \lesssim  \dual{\hyp \widetilde v}{\widetilde v}_\Gamma =
 \dual{\slp\widetilde\phi}{\widetilde\phi}_\Gamma =
   \norm{\widetilde\phi}\slp^2.
\end{align}
Recall that $\widetilde\phi = \phi-\phi_0$.
With~\eqref{eq:slpLMLD2}, the triangle
inequality yields 
\begin{align}\label{eq:slpLMLD8}
  \norm{\phi_0}\slp^2 + \sum_{\ell=0}^L \sum_{z\in\widetilde\NN_\ell^\Gamma} 
  \norm{\phi_z^\ell}\slp^2 \lesssim  \norm{\phi_0}\slp^2 +
  \norm{\widetilde\phi}\slp^2 \lesssim \norm{\phi_0}\slp^2 +
  \norm{\phi}\slp^2 \lesssim \norm{\phi}\slp^2,
\end{align}
where the hidden constants depend only on $\Gamma$, the initial triangulation
$\TT_0^\Gamma$, as well as the chosen mesh-refinement strategy.
By means of Lemma~\ref{lem:stabledec}~(i), this proves the lower bound
in~\eqref{eq:slpLMLD}.
\end{proof}

\begin{proof}[Proof of Proposition~\ref{thm:slpLMLD}, upper bound in~\eqref{eq:slpLMLD}]
Recall the unique splitting $\phi = \phi_0 + \widetilde\phi$
from~\eqref{eq:splitPhi}. Let $\sum_{\ell=0}^L \sum_{z\in\widetilde\NN_\ell^\Gamma} \phi_z^\ell
= \widetilde\phi$ denote an arbitrary splitting of $\widetilde\phi \in 
\sum_{\ell=0}^L \sum_{z\in\widetilde\NN_\ell^\Gamma} \YY_z^\ell$. Note that
$\phi_z^\ell = \alpha_z^\ell \chi_z^\ell$ for some $\alpha_z^\ell\in\R$. We define $\widetilde v$ as
\begin{align}\label{eq:slpLMLDup3}
  \widetilde v := \sum_{\ell=0} \sum_{z\in\widetilde\NN_\ell^\Gamma} v_z^\ell \quad\text{with }
  v_z^\ell := \alpha_z^\ell \zeta_z^\ell
\end{align}
and stress that $\widetilde\phi = (\widetilde v)'$ as well as $\phi_z^\ell = (v_z^\ell)'$.
The abstract result~\cite[Lemma~3.1]{zhang92} states
\begin{align}\label{eq:proj:evmax}
  \evmax(\PAShyp) = \max_{v\in\ZZ^L}
  \frac{\norm{v}\hyp^2}{\min\limits_{\sum_{\ell=0}^L \sum_{z\in\widetilde\NN_\ell^\Gamma} v_z^\ell = v}
  \sum_{\ell=0}^L \sum_{z\in\widetilde\NN_\ell^\Gamma} \norm{v_z^\ell}\hyp^2},
\end{align}
since $\PAShyp$ is a finite sum of symmetric projections.
From~Lemma~\ref{lem:hypAS},
we get uniform boundedness of the Rayleigh quotient $\dual{\PAShyp v}v_\hyp /
\norm{v}\hyp^2\leq \c{hyp:ub}$. Thus, $\evmax(\PAShyp)\leq \c{hyp:ub} < \infty$ is uniformly bounded.
Then,~\eqref{eq:proj:evmax} yields
\begin{align}\label{eq:slpLMLDup2}
  \norm{\widetilde v}\hyp^2 \lesssim \sum_{\ell=0}^L \sum_{z\in\widetilde\NN_\ell^\Gamma}
  \norm{v_z^\ell}\hyp^2.
\end{align}
Together with Maue's formula~\eqref{eq:nedelec}, the
definition~\eqref{eq:slpLMLDup3}, and the norm equivalence 
\begin{align*}
  \norm{v_z^\ell}\hyp^2 \simeq \norm{v_z^\ell}{H^{1/2}(\Gamma)}^2 = \norm{v_z^\ell}{\widetilde
H^{1/2}(\omega_\ell(z))}^2 \simeq \dual{\hyp v_z^\ell}{v_z^\ell}_\Gamma 
= \dual{\slp \phi_z^\ell}{\phi_z^\ell}_\Gamma, 
\end{align*}
we infer 
\begin{align*}
  \norm{\widetilde\phi}\slp^2 =
  \dual{\slp\widetilde\phi}{\widetilde\phi}_\Gamma = 
  \dual{\hyp\widetilde v}{\widetilde v}_\Gamma \leq \norm{\widetilde v}\hyp^2
  \lesssim \sum_{\ell=0}^L \sum_{z\in\widetilde\NN_\ell^\Gamma} \norm{v_z^\ell}\hyp^2
  \simeq \sum_{\ell=0}^L \sum_{z\in\widetilde\NN_\ell^\Gamma}
  \norm{\phi_z^\ell}\slp^2.
\end{align*}
Finally, $\phi = \phi_0 + \widetilde\phi$ yields
\begin{align}\label{eq:slpLMLDup5}
\begin{split}
  \norm\phi\slp^2 &\lesssim \norm{\phi_0}\slp^2 +
  \norm{\widetilde\phi}\slp^2 \lesssim \norm{\phi_0}\slp^2 + 
  \sum_{\ell=0}^L \sum_{z\in\widetilde\NN_\ell^\Gamma} \norm{\phi_z^\ell}\slp^2,
\end{split}
\end{align}
where the hidden constants depend only on $\Gamma$, the initial
triangulation $\TT_0^\Gamma$, and the chosen mesh-refinement strategy.
Lemma~\ref{lem:stabledec} (ii) and~\eqref{eq:slpLMLDup5} prove the upper bound
in~\eqref{eq:slpLMLD}.

\end{proof}

\begin{proof}[Proof of Theorem~\ref{thm:precond:slp}]
Recall that $\psi_{T_m}^\ell \in\YY^\ell$ denotes the characteristic function of
$T_m^\ell \in\TT_\ell^\Gamma$.
First, we prove the relation
\begin{align*}
  \PASslp \phi = \sum_{m=1}^{M_L} (\PPrec_\slp^{-1}\AA_\slp\pphi)_m \psi_{T_m^L}
  \quad\text{for all } \phi = \sum_{j=1}^{M_L} (\pphi)_j \psi_{T_j^L} \in \YY^L.
\end{align*}
Recall $\PASslp = \precslp^{00} + \sum_{\ell=0}^L
\sum_{z\in\widetilde\NN_\ell^\Gamma} \precslp_z^\ell$ from~\eqref{def:slp:as}
and $\onemat\in\R^{M_L}$ with $(\onemat)_j = 1$ for all $j=1,\dots,M_L$.
By definition~\eqref{def:slp:proj} of $\precslp^{00}$, it follows with $D=\norm{1}\slp^2$
\begin{align}\label{eq:precslp:const}
  \precslp^{00} \phi = D^{-1}\dual{\phi}1_\slp \,1 = D^{-1}
  \dual{\pphi}{\onemat}_{\AA_\slp} = \sum_{m=1}^{M_L}  (\onemat
  D^{-1} \onemat^T \AA_\slp \pphi)_m \psi_{T_m^L}.
\end{align}
Moreover, from~\eqref{def:slp:proj} we also infer
\begin{align}\label{eq:precslp:node}
  \precslp_z^\ell\phi = \norm{\chi_z^\ell}\slp^{-2}
  \dual{\phi}{\chi_z^\ell}_\slp \, \chi_z^\ell \quad\text{for all }
  z\in\widetilde\NN_\ell^\Gamma.
\end{align}
With the definition of $\widetilde\Haar^\ell$ and $\embedPmat^\ell$ from
Section~\ref{sec:main:LMLD}, each Haar basis function $\chi_z^\ell \in
\widetilde\ZZ^\ell$ can be represented as
\begin{align}\label{eq:haar:repr}
  \chi_{z_k}^\ell = \sum_{j=1}^{M_\ell} ( \widetilde\Haar^\ell)_{jk}
  \psi_{T_j^\ell} = \sum_{j=1}^{M_\ell} \sum_{m=1}^{M_L}
  (\widetilde\Haar^\ell)_{jk} (\embedPmat^\ell)_{mj} \psi_{T_m^L}
  = \sum_{m=1}^{M_L} (\embedPmat^\ell \widetilde\Haar^\ell)_{mk} \psi_{T_m^L}.
\end{align}
Thus, 
\begin{align*}
  \dual{\phi}{\chi_{z_k}^\ell}_\slp &= \sum_{m=1}^{M_L}
  (\embedPmat^\ell\widetilde\Haar^\ell)_{mk}
  \dual{\phi}{\psi_{T_m^L}}_{\slp} = \sum_{m=1}^{M_L}
  ( (\widetilde\Haar^\ell)^T (\embedPmat^\ell)^T)_{km} (\AA_\slp\pphi)_m \\
  &= ((\widetilde\Haar^\ell)^T (\embedPmat^\ell)^T \AA_\slp\pphi)_k.
\end{align*}
Furthermore, the last identity together with~\eqref{eq:precslp:node},
\eqref{eq:haar:repr}, and $\norm{\chi_{z_k}^\ell}\slp^2 =
(\widetilde\DD_\hyp^\ell)_{kk}$ show
\begin{align}
  \precslp_{z_k}^\ell \phi = \sum_{m=1}^{M_L} (\embedPmat^\ell
  \widetilde\Haar^\ell)_{mk} (\widetilde\DD_\hyp^\ell)^{-1}_{kk} 
  ((\widetilde\Haar^\ell)^T (\embedPmat^\ell)^T \AA_\slp\pphi)_k \psi_{T_m^L}.
\end{align}
Summing the last terms over all $k = 1,\dots,\widetilde N_\ell^\Gamma := \#
\widetilde\NN_\ell^\Gamma$ and $\ell=0,\dots,L$ yields
with~\eqref{eq:precslp:const}
\begin{align}
\begin{split}
  \PASslp \phi &= \sum_{m=1}^{M_L} (\onemat D^{-1} \onemat^T \AA_\slp\pphi)_m
  \psi_{T_m^L} +  \sum_{\ell=0}^L \sum_{m=1}^{M_L} 
   (\embedPmat^\ell \widetilde\Haar^\ell (\widetilde\DD_\hyp^\ell)^{-1}
   (\widetilde\Haar^\ell)^T (\embedPmat^\ell)^T \AA_\slp\pphi)_m \psi_{T_m^L} 
   \\
   &= \sum_{m=1}^{M_L} (\PPrec_\slp^{-1} \AA_\slp\pphi)_m \psi_{T_m^L}. 
\end{split}
\end{align}
The last identity together with Proposition~\ref{thm:slpLMLD} implies
\begin{align}\label{eq:precslp:equiv}
  \dual{\PPrec_\slp^{-1}\AA_\slp \pphi}\pphi_{\AA_\slp} =
  \dual{\PASslp\phi}\phi_\slp \simeq \norm{\phi}\slp^2 = \norm\pphi{\AA_\slp}^2,
\end{align}
where the hidden constants depend only on $\Gamma$, the initial triangulation
$\TT_0^\Gamma$, as well as on the chosen mesh-refinement.
Finally, by setting $\pphi = \AA_\slp^{-1/2}(\AA_\slp^{-1/2}\PPrec_\slp\AA_\slp^{-1/2})^{1/2}
\AA_\slp^{1/2} {\boldsymbol\Psi}$
in~\eqref{eq:precslp:equiv}, we get
\begin{align*}
  \dual{\PPrec_\slp \boldsymbol\Psi}{\boldsymbol\Psi}_2
  \simeq \dual{\AA_\slp \boldsymbol\Psi}{\boldsymbol\Psi}_2 \quad\text{for all }
  \boldsymbol\Psi \in\R^{M_L},
\end{align*}
which concludes the proof.
\end{proof}


\section{Proof of Theorem~\ref{thm:main}}\label{sec:proof}

\noindent
Basically, we follow the lines of the proof of~\cite[Theorem~5.2]{ms98}, which
was stated for the symmetric coupling and a block-diagonal
preconditioner based on a
hierarchical basis decomposition of the underlying discrete spaces. Here, we adapt the proof to the
non-symmetric Johnson-N\'ed\'elec coupling.

For the analysis of the proposed block-diagonal preconditioner, we define the
operator $\pform$, which can be interpreted as a preconditioning form of the
operator $\jn$.
The next result directly follows from the properties of the operator
$\fem$ from Lemma~\ref{lemma:A} and the properties of the simple-layer integral operator
$\slp$.

\begin{lemma}\label{lemma:B}
  For $(u,\phi),(v,\psi)\in\HH$, define
  \begin{align*}
    \dual{\pform(u,\phi)}{(v,\psi)} := \dual{\fem u}v +\dual\psi{\slp\phi}_\Gamma.
  \end{align*}
  Then, the operator $\pform : \HH\to\HH^*$ is linear, symmetric, continuous, and elliptic, and the constants.
  \begin{align*}
    c_\pform := \inf_{\bignull\neq
    (u,\phi)\in\HH}\frac{\dual{\pform(u,\phi)}{(u,\phi)}}{\norm{(u,\phi)}\HH^2}
    \quad\text{and}\quad  C_\pform := \sup_{\bignull\neq (u,\phi)\in\HH}\frac{\norm{\pform(u,\phi)}{\HH^*}}{\norm{(u,\phi)}\HH}
  \end{align*}
  satisfy $0<c_\pform\le C_\pform<\infty$ and depend only on
  $c_\fem$ and $C_\fem$ from~\eqref{eq:const:A} as well as on $\Omega$.\qed
\end{lemma}

The following auxiliary result is explicitly stated for the Johnson-N\'ed\'elec coupling
and also used in~\cite{ms98} for the symmetric coupling
accordingly.

\begin{lemma}\label{lemma:discretejn}
Let $\HH^L \subset \HH = H^1(\Omega)\times H^{-1/2}(\Gamma)$ be a finite
dimensional subspace.
Let $\jn_L, \pform_L : \HH^L \to (\HH^L)^*$ denote the operators $\jn,\pform$ restricted to the
discrete space $\HH^L$, i.e.
\begin{align*}
  \dual{\jn_L\uuu_L}{\vvv_L} &:= \dual{\jn \uuu_L}{\vvv_L} \\
  \dual{\pform_L\uuu_L}{\vvv_L} &:= \dual{\pform \uuu_L}{\vvv_L}
\end{align*}
for all $\uuu_L,\vvv_L\in\HH^L$. Define $\tmp_L := \jn_L^*
\pform_L^{-1} \jn_L
: \HH^L\to(\HH^L)^*$. 
Then, the Galerkin matrix $\tmpmat$ of $\tmp_L$ with respect to the basis
$\{\varphi_j\}_{j=1}^{N_L+M_L}$ of $\HH^L$, i.e. $(\tmpmat)_{jk} =
\dual{\tmp_L\varphi_k}{\varphi_j}$ for $j,k=1,\dots,N_L+M_L$, satisfies
\begin{align}
  \tmpmat = \AA_\jn^T \AA_\pform^{-1} \AA_\jn.
\end{align}
\end{lemma}

\begin{proof}
Let $\{\ee^j\}_{j=1}^{N_L+M_L}$ denote the canonical basis of
$\R^{N_L+M_L}$ with
$\dual{\ee^j}{\ee^k}_2 = \delta_{jk}$.
The matrix entry of the $j$-th row and $k$-th column is given by
\begin{align*}
  (\tmpmat)_{jk} = \dual{\jn_L^*\pform_L^{-1}\jn_L \varphi_k}{\varphi_j} =
  \dual{\pform_L^{-1} \jn_L\varphi_k}{\jn_L\varphi_j}.
\end{align*}
Let $\ww_k := \pform_L^{-1} \jn_L\varphi_k$ 
and note that by definition of the inverse of the \emph{discrete} operator $\pform_L$,
$\ww_k$ satisfies
\begin{align}\label{eq:discretejn:proof1}
  \dual{\pform_L \ww_k}{\varphi_m} = \dual{\jn_L\varphi_k}{\varphi_m}
  \quad m = 1,\dots,N_L+M_L.
\end{align}
Note that $\dual{\jn_L\varphi_k}{\varphi_m} = (\AA_\jn)_{m k} = (\AA_\jn\ee^k)_m$.
Together with the basis representation $\ww_k = \sum_{j=1}^{N_L+M_L} (\WW)_j
\varphi_j$, the Galerkin formulation~\eqref{eq:discretejn:proof1} of $\ww_k$ is thus equivalent to
\begin{align*}
  \AA_\pform \WW = \AA_\jn \ee^k.
\end{align*}
By choice of $\ww_k$, we get
\begin{align*}
  (\tmpmat)_{jk} &= \dual{\ww_k}{\jn_L \varphi_j} =
  \sum_{m=1}^{N_L+M_L}
  (\WW)_m \dual{\varphi_m}{\jn_L \varphi_j} = \sum_{m=1}^{N_L+M_L}
  \big(\AA_\pform^{-1}\AA_\jn\ee^k\big)_m
  \big(\AA_\jn\ee^j\big)_m \\
  &= \dual{\AA_\pform^{-1}\AA_\jn\ee^k}{\AA_\jn\ee^j}_2 = \dual{\AA_\jn^T
  \AA_\pform^{-1} \AA_\jn\ee^k}{\ee^j}_2 \quad\text{for all
  }j,k=1,\dots,N_L+M_L.
\end{align*}
This concludes the proof.
\end{proof}

\begin{proof}[Proof of Theorem~\ref{thm:main}]
We use the following result on the reduction of the relative residual in the
preconditioned GMRES
Algorithm~\ref{alg:gmres}, which can be found, e.g., in~\cite[Section~3]{heuerstep98}:
Due to~\cite{ees83,saad}, the $j$-th residuum from the preconditioned GMRES
Algorithm~\ref{alg:gmres} is bounded by
\begin{subequations}\label{eq:gmres:bound}
\begin{align}
  \norm{\RR_j}{\PPrec_\jn} \leq \big( 1-\alpha^2/\beta^2\big)^{j/2}
  \norm{\RR_0}{\PPrec_\jn},
\end{align}
with constants
\begin{align}\label{eq:gmres:bounda}
  \alpha
  &:=\inf_{\UU\neq\bignull}\frac{\dual{\PPrec_\jn^{-1}\AA_\jn\UU}{\UU}_{\PPrec_\jn}}{\norm{\UU}{\PPrec_\jn}^2}
  \leq \inf_{\UU\neq\bignull}
  \frac{\norm{\PPrec_\jn^{-1}\AA_\jn\UU}{\PPrec_\jn}}{\norm{\UU}{\PPrec_\jn}} =
  \norm{(\PPrec_\jn^{-1}\AA_\jn)^{-1}}{\PPrec_\jn}^{-1}, \\
  \label{eq:gmres:boundb}
  \beta  &:=\sup_{\UU\neq\bignull}\frac{\norm{\PPrec_\jn^{-1}\AA_\jn\UU}{\PPrec_\jn}}{\norm{\UU}{\PPrec_\jn}}
   = \norm{\PPrec_\jn^{-1}\AA_\jn}{\PPrec_\jn},
\end{align}
\end{subequations}
when $\dual\cdot\cdot_{\PPrec_\jn} := \dual{\PPrec_\jn(\cdot)}\cdot_2$ is used as inner product
in the preconditioned GMRES algorithm. 
We also refer to~\cite{sarkisszyld} for a discussion on preconditioned
GMRES methods using different inner products.

Due to~\eqref{eq:gmres:bound}, $\cond_{\PPrec_\jn}(\PPrec_\jn^{-1}\AA_\jn) \leq
\beta/\alpha$ and we have to provide a lower bound
for~\eqref{eq:gmres:bounda} and an upper bound for~\eqref{eq:gmres:boundb}.
Recall that the preconditioner matrix $\PPrec_\jn$ and the Galerkin matrix
$\AA_\pform$ of $\pform$ have the form
\begin{align*}
  \PPrec_\jn = \begin{pmatrix}
    \PPrec_\fem & \bignull \\
    \bignull & \PPrec_\slp
  \end{pmatrix},
  \qquad \AA_\pform = 
  \begin{pmatrix}
    \AA_\fem & \bignull \\
    \bignull & \AA_\slp
  \end{pmatrix}.
\end{align*}
Define $d_\pform:= \min\{d_\fem,d_\slp\}$, $D_\pform:=\max\{D_\fem,D_\slp\}$.
From Theorem~\ref{thm:precond:fem} and Theorem~\ref{thm:precond:slp}, it follows
\begin{align}\label{eq:P_ineq}
  d_\pform \dual{\PPrec_\jn\VV}{\VV}_2 \leq \dual{\AA_\pform\VV}{\VV}_2 \leq
  D_\pform \dual{\PPrec_\jn\VV}{\VV}_2 \quad\text{for all }\VV\in\R^{N_L+M_L},
\end{align}
which is equivalent to
\begin{align}\label{eq:P_ineq2}
  d_\pform \dual{\AA_\pform^{-1}\UU}{\UU}_2 \leq \dual{\PPrec_\jn^{-1}\UU}{\UU}_2 \leq
  D_\pform \dual{\AA_\pform^{-1}\UU}{\UU}_2 \quad\text{for all }\UU\in\R^{N_L+M_L}.
\end{align}
Here, the equivalence of~\eqref{eq:P_ineq}--\eqref{eq:P_ineq2} follows from
the choice 
$\UU = \AA_\pform^{1/2} (\AA_\pform^{-1/2} \PPrec_\jn
\AA_\pform^{-1/2})^{1/2} \AA_\pform^{1/2}\VV$
resp.\ 
$\VV = \AA_\pform^{-1/2} (\AA_\pform^{-1/2} \PPrec_\jn
\AA_\pform^{-1/2})^{-1/2} \AA_\pform^{-1/2}\UU$
and elementary calculations.
Setting $\UU = \AA_\pform\WW$ in~\eqref{eq:P_ineq2} and since $\PPrec_\jn^{-1}
\AA_\pform$ is symmetric with respect to $\dual\cdot\cdot_{\AA_\pform}$ yields
\begin{align}\label{eq:P_ineq3}
  d_\pform \leq \evmin(\PPrec_\jn^{-1}\AA_\pform) \leq
  \dual{\PPrec_\jn^{-1}\AA_\pform\WW}{\WW}_{\AA_\pform} / \norm\WW{\AA_\pform}^2 \leq
  \evmax(\PPrec_\jn^{-1}\AA_\pform) \leq D_\pform.
\end{align}

For given $\UU = (\xx,\pphi)^T \in \R^{N_L+M_L}$, let $\uu_L = (u_L,\phi_L)
\in\HH^L$ denote the corresponding function.
We start to prove a lower bound for~\eqref{eq:gmres:bounda}. 
Lemma~\ref{lemma:A}, Lemma~\ref{lemma:B}, and~\eqref{eq:P_ineq} yield
\begin{align*}
  \dual{\PPrec_\jn^{-1}\AA_\jn\UU}{\UU}_{\PPrec_\jn} &= \dual{\AA_\jn
  \UU}{\UU}_2 = \dual{\jn\uu_L}{\uu_L} \geq c_\jn \norm{\uu_L}\HH^2 \geq
  c_\jn C_\pform^{-1} \dual{\pform\uu_L}{\uu_L} \\
  &= c_\jn C_\pform^{-1}\dual{\AA_\pform\UU}\UU_2 
  = c_\jn C_\pform^{-1} d_\pform \dual{\PPrec_\jn\UU}\UU_2
  = c_\jn C_\pform^{-1} d_\pform \norm\UU{\PPrec_\jn}^2.
\end{align*}
Thus, $\widetilde\alpha := c_\jn C_\pform^{-1} d_\pform = c_\jn C_\pform^{-1} \min\{d_\fem,d_\slp\}$ is a lower
bound for $\alpha$ from~\eqref{eq:gmres:bounda}.

It remains to prove an upper bound for $\beta$ from~\eqref{eq:gmres:boundb}.
With~\eqref{eq:P_ineq2} and the discrete operator $\tmp_L = \jn_L^*
\pform_L^{-1} \jn_L$ with Galerkin matrix $\tmpmat$ from Lemma~\ref{lemma:discretejn}, we infer
\begin{align*}
  \norm{\PPrec_\jn^{-1}\AA_\jn\UU}{\PPrec_\jn}^2 
  = \dual{\AA_\jn\UU}{\PPrec_\jn^{-1}\AA_\jn\UU}_2
  \leq D_\pform \dual{\AA_\pform^{-1}\AA_\jn\UU}{\AA_\jn\UU}_2 
  &= D_\pform \dual{\tmpmat \UU}\UU_2 \\
  &= D_\pform \dual{\tmp_L\uu_L}{\uu_L}.
\end{align*}
Moreover, it holds
\begin{align*}  
 \dual{\tmp_L\uu_L&}{\uu_L} 
  = \dual{\jn_L^* \pform_L^{-1} \jn_L \uu_L}{\uu_L}
  = \dual{\pform_L^{-1} \jn_L \uu_L}{\jn_L \uu_L} 
  = \dual{\pform_L^{-1} \jn_L \uu_L}{\jn \uu_L} \\
  &\leq \norm{\jn\uu_L}{\HH^*} \norm{\pform_L^{-1} \jn_L\uu_L}\HH
  \leq C_\jn \norm{\uu_L}\HH \norm{\pform_L^{-1} \jn_L\uu_L}\HH.
\end{align*}
Note that $\ww_L := \pform_L^{-1} \jn_L \uu_L$ is the Galerkin solution of
\begin{align*}
  \dual{\pform \ww_L}{\vv_L} = \dual{\jn \uu_L}{\vv_L} \quad\text{for all }\vv_L
  \in\HH^L.
\end{align*}
Therefore, we can estimate the norm of $\ww_L$ by
\begin{align*}
  c_\pform\,\norm{\ww_L}\HH^2 
  \leq \dual{\pform \ww_L}{\ww_L} 
  = \dual{\jn \uu_L}{\ww_L} 
  \leq \norm{\jn\uu_L}{\HH^*}\norm{\ww_L}\HH 
  \leq C_\jn \norm{\uu_L}\HH\norm{\ww_L}\HH.
\end{align*}
With~\eqref{eq:P_ineq}, this altogether gives
\begin{align*}
  \norm{\PPrec_\jn^{-1}\AA_\jn\UU}{\PPrec_\jn}^2 &\leq D_\pform c_\pform^{-1}
  C_\jn^2 \norm{\uu_L}\HH^2\leq D_\pform c_\pform^{-2} C_\jn^2 
  \dual{\pform \uu_L}{\uu_L} = D_\pform c_\pform^{-2} C_\jn^2 \dual{\AA_\pform \UU}\UU_2 \\
  & \leq D_\pform^2 c_\pform^{-2} C_\jn^2 \dual{\PPrec_\jn \UU}\UU_2
  = D_\pform^2 c_\pform^{-2} C_\jn^2 \norm\UU{\PPrec_\jn}^2.
\end{align*}
Then, $\widetilde\beta := D_\pform c_\pform^{-1} C_\jn = \max\{D_\fem,D_\slp\}
c_\pform^{-1} C_\jn$ is an upper bound for $\beta$ from~\eqref{eq:gmres:boundb}.

Finally, the definition
\begin{align*}
  q_{\rm gmres} := (1-\widetilde\alpha^2/\widetilde\beta^2)^{1/2}  = \left(1-
  \Big(\frac{c_\jn c_\pform \min\{d_\fem,d_\slp\}}{C_\jn C_\pform
  \max\{D_\fem,D_\slp\}}\Big)^2 \right)^{1/2}
\end{align*}
concludes the proof.
\end{proof}

\begin{remark}\label{rem:cond}
Note that the last proof unveils
\begin{align}
  \cond_{\PPrec_\jn}(\PPrec_\jn^{-1}\AA_\jn) \lesssim
  \frac{\evmax(\PPrec_\jn^{-1}\AA_\pform)}{\evmin(\PPrec_\jn^{-1}\AA_\pform)} =
  \cond_{\PPrec_\jn}(\PPrec_\jn^{-1}\AA_\pform) = \cond_{\AA_\pform}(\PPrec_\jn^{-1}\AA_\pform).
\end{align}
This result can be obtained by replacing $d_\pform$ with
$\evmin(\PPrec_\jn^{-1}\AA_\pform)$ resp. $D_\pform$ with
$\evmax(\PPrec_\jn^{-1}\AA_\pform)$ in the proof of Theorem~\ref{thm:main}.
\end{remark}


\section{Extension to other coupling methods and further remarks}\label{sec:extensions}

\subsection{Symmetric coupling}
The model problem~\eqref{eq:model} can equivalently be reformulated by means of the
symmetric coupling~\cite{costabel,han90}:
Find $(u,\phi)\in\HH := H^1(\Omega)\times H^{-1/2}(\Gamma)$ such that
\begin{align}\label{eq:sym:varform}
\begin{split}
  \dual{\material\nabla u}{\nabla v}_\Omega + \dual{\hyp u}v_\Gamma +
  \dual{(\dlp'-\tfrac12)\phi}v_\Gamma &= \dual{f}v_\Omega + \dual{\phi_0 + \hyp
  u_0}v_\Gamma, \\
  \dual\psi{(\tfrac12-\dlp)u+\slp\phi}_\Gamma &=
  \dual\psi{(\tfrac12-\dlp)u_0}_\Gamma.
\end{split}
\end{align}
for all $(v,\psi)\in\HH$. Analogously to the
Johnson-N\'ed\'elec coupling~\eqref{eq:jn:staboperator}--\eqref{eq:jn:disrete},
we define the operator $\sym : \HH\to\HH^*$ resp. the linear
functional $\symRHS\in\HH^*$ for an equivalent operator 
formulation
\begin{align*}
  \sym(u,\phi) = \symRHS
\end{align*}
of the symmetric
coupling~\eqref{eq:sym:varform} by
\begin{align*}
  \dual{\sym (u,\phi)}{(v,\psi)} &:= \dual{\material\nabla u}{\nabla v}_\Omega +
  \dual{\hyp u}v_\Gamma + \dual{(\dlp'-\tfrac12)\phi}v_\Gamma
  + \dual\psi{(\tfrac12-\dlp)u+\slp\phi}_\Gamma
  \\
  &\qquad + \dual{1}{(\tfrac12-\dlp)u+\slp\phi}_\Gamma
  \dual{1}{(\tfrac12-\dlp)v+\slp\psi}_\Gamma, \\
  \dual{\symRHS}{(u,\phi)} &:= \dual{f}v_\Gamma + \dual{\phi_0 + \hyp
  u_0}v_\Gamma   + \dual\psi{(\tfrac12-\dlp)u_0}_\Gamma \\
  &\qquad + \dual{1}{(\tfrac12-\dlp)u_0}_\Gamma
  \dual{1}{(\tfrac12-\dlp)v+\slp\psi}_\Gamma,
\end{align*}
for all $(u,\phi),(v,\psi)\in\HH$.
We stress that Lemma~\ref{lem:jn:stab} also holds for the symmetric coupling
with $(\jn,F)$ replaced by $(\sym,\symRHS)$. The following
result can be found in~\cite[Section~5]{affkmp}.

\begin{lemma}
Lemma~\ref{lem:jn:stab} holds accordingly for the symmetric coupling,
where~\eqref{eq:contraction} is replaced by $c_\material > 0$.\qed
\end{lemma}
Let $\AA_\hyp$ denote the Galerkin matrix of the hypersingular integral operator
with respect to the nodal basis of $\XX^L$,
i.e. $(\AA_\hyp)_{jk} = \dual{\hyp \eta_{z_k}^L}{\eta_{z_j}^L}_\Gamma$ for all
$j,k=1,\dots,N_L$.
The Galerkin matrix $\AA_\sym$ of the operator $\sym$ reads in matrix block form
\begin{align*}
  \AA_\sym = \begin{pmatrix}
    \AA_\material + \AA_\hyp & \Kmat^T - \tfrac12\Mmat^T \\
    \tfrac12\Mmat - \Kmat & \AA_\slp
  \end{pmatrix} + \stabmat\stabmat^T.
\end{align*}
We use the block-diagonal preconditioner
\begin{align*}
  \PPrec_\sym := \begin{pmatrix}
    \PPrec_{\widetilde\fem} & \bignull \\
    \bignull & \PPrec_\slp
  \end{pmatrix},
\end{align*}
which is similar to the one for the
Johnson-N\'ed\'elec coupling. 
Here, $\widetilde\fem : H^1(\Omega)\to (H^1(\Omega))^*$ is defined as
$\dual{\widetilde\fem u}v := \dual{\fem u}v + \dual{\hyp u}v_\Gamma$ and
$\PPrec_{\widetilde\fem}$ is defined as $\PPrec_\fem$ with the diagonals of
$\AA_\fem$ replaced by the diagonals of the Galerkin matrix of
$\widetilde\fem$.
We seek for a solution of the preconditioned system
\begin{align}\label{eq:sym:precondsystem}
  \PPrec_\sym^{-1} \AA_\sym \UU = \PPrec_\sym^{-1} \widetilde\FF,
\end{align}
where $\widetilde\FF$ denotes the discretization of the right-hand side
$\widetilde F$.
The following theorem is proved along the lines of Section~\ref{sec:proof} with the obvious
modifications.

\begin{theorem}
Theorem~\ref{thm:main} holds accordingly for the symmetric coupling.\qed
\end{theorem}

\subsection{One-equation Bielak-MacCamy coupling}
The model problem~\eqref{eq:model} can equivalently be rewritten by means of the one-equation Bielak-MacCamy
coupling~\cite{bmc} which can be seen as the ``transposed'' Johnson-N\'ed\'elec
coupling: Find $(u,\phi)\in\HH$ such that
\begin{align}\label{eq:bmc:varform}
\begin{split}
  \dual{\material\nabla u}{\nabla v}_\Omega + \dual{(\tfrac12-\dlp')\phi}v_\Gamma
  &= \dual{f}v_\Omega + \dual{\phi_0}v_\Gamma, \\
  \dual{\psi}{\slp\phi-u}_\Gamma &= -\dual\psi{u_0}_\Gamma,
\end{split}
\end{align}
for all $(v,\psi)\in\HH$.
Analogously to the Johnson-N\'ed\'elec coupling~\eqref{eq:jn:staboperator}--\eqref{eq:jn:disrete}, we
define the operator $\bmc : \HH\to\HH^*$ and the linear functional
$\bmcRHS\in\HH^*$ for an equivalent operator
formulation
\begin{align*}
  \sym(u,\phi) = \symRHS
\end{align*}
of the Bielak-MacCamy
coupling~\eqref{eq:bmc:varform} by
\begin{align*}
  \dual{\bmc (u,\phi)}{(v,\psi)} &:= 
  \dual{\material\nabla u}{\nabla v}_\Omega + \dual{(\tfrac12-\dlp')\phi}v_\Gamma 
  + \dual{\psi}{\slp\phi-u}_\Gamma \\
  &\qquad + \dual{1}{\slp\phi-u}_\Gamma \dual{1}{\slp\psi-v}_\Gamma , \\
  \dual{\bmcRHS}{(v,\psi)} &:= \dual{f}v_\Omega + \dual{\phi_0}v_\Gamma 
  -\dual\psi{u_0}_\Gamma - \dual{1}{u_0}_\Gamma \dual{1}{\slp\psi-v}_\Gamma,
\end{align*}
for all $(u,\phi),(v,\psi)\in\HH$. The following result is found
in~\cite[Section~3]{affkmp}.

\begin{lemma}
Provided~\eqref{eq:contraction}, Lemma~\ref{lem:jn:stab} holds accordingly for the Bielak-MacCamy coupling. \qed
\end{lemma}
The Galerkin matrix $\AA_\bmc$ of the operator $\bmc$ reads in matrix block form
\begin{align*}
  \AA_\bmc = \begin{pmatrix}
    \AA_\material & \tfrac12 \Mmat^T - \Kmat^T \\
    -\Mmat & \AA_\slp
  \end{pmatrix}
  +\stabmatBmc\stabmatBmc^T,
\end{align*}
where the (column) vector $\stabmatBmc$ is defined componentwise by
$(\stabmatBmc)_j := -\dual{1}{\eta_{z_j}^L}_\Gamma$ for $j=1,\dots,N_L$ and
$(\stabmatBmc)_{j+N_L} := \dual{1}{\slp\psi_{T_j}}_\Gamma$ for $j=1,\dots,M_L$. We use
the same block-diagonal preconditioner~\eqref{eq:precond_mat} as for the
Johnson-N\'ed\'elec coupling. The following theorem is proved along the lines of Section~\ref{sec:proof} with the obvious
modifications.

\begin{theorem}
Theorem~\ref{thm:main} holds accordingly for the Bielak-MacCamy coupling. \qed
\end{theorem}

\subsection{Further remarks}
The analysis in Section~\ref{sec:proof} depends only on the spectral
estimates~\eqref{eq:specest:fem} and~\eqref{eq:specest:slp}. Therefore, the
multilevel additive Schwarz preconditioners $\PPrec_\fem$ and $\PPrec_\slp$ can be replaced by
any preconditioners $\widetilde\PPrec_\fem$ and $\widetilde\PPrec_\slp$ such that
\begin{align}\label{eq:rem:specest}
  \dual{\widetilde\PPrec_\fem \xx}{\xx}_2 \simeq \dual{\AA_\fem\xx}{\xx}_2
  \quad\text{and}\quad 
  \dual{\widetilde\PPrec_\slp\pphi}{\pphi}_2 \simeq \dual{\AA_\slp\pphi}{\pphi}_2
\end{align}
holds for all $\xx\in\R^{N_L},\pphi\in\R^{M_L}$.
The reduction constant $q_{\rm GMRES}$ from Theorem~\ref{thm:main} then depends
on the equivalence constants in~\eqref{eq:rem:specest}.
Preferably, the preconditioners $\widetilde\PPrec_\fem$ and
$\widetilde\PPrec_\slp$ should be chosen such that these constants are
independent of mesh-related quantities as is the case for the local multilevel
additive Schwarz preconditioners considered here.

The techniques presented in this work may also apply for the (quasi-)symmetric
Bielak-MacCamy coupling~\cite{bmc}. A stability analysis of this coupling method
can be found in~\cite{ghs09}.

It is also possible to apply our analysis to other
model problems, e.g., transmission problems for linear elasticity.
We stress that our approach requires a (possibly non-symmetric) positive definite Galerkin
matrix, associated to the coupling method.
For Lam\'e-type problems, this can be ensured by stabilization, where a result
analogously to Lemma~\ref{lem:jn:stab} remains valid~\cite{FBlame}.


\section{Numerical examples}\label{sec:examples}

\begin{figure}[t]
  \begin{center}
  \psfrag{x}[c][c]{\tiny $x$}
  \psfrag{y}[c][c]{\tiny $y$}
  \includegraphics[width=0.8\textwidth]{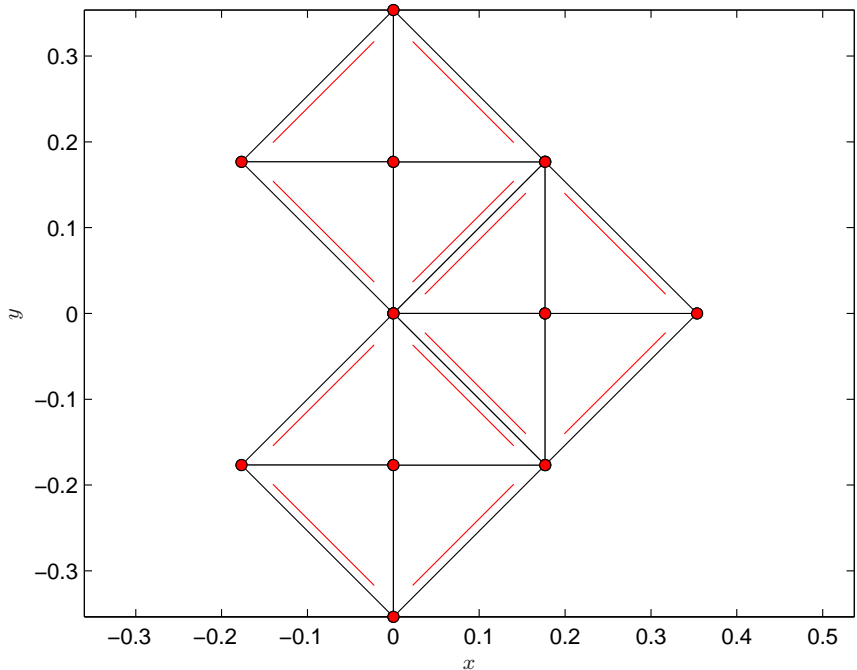}
  \caption{L-shaped domain $\Omega$ with $\diam(\Omega)<1$ and initial volume triangulation
  $\TT_0^\Omega$. The initial triangulation $\TT_0^\Gamma$ of the boundary is given by
  the restriction $\TT_0^\Omega|_\Gamma$ of the volume triangulation on the
  boundary. The initial triangulations consist of $\#\TT_0^\Omega = 12$ resp.
  $\# \TT_0^\Gamma = 8$ elements. Red lines indicate the reference edges for the
  newest vertex bisection of the initial volume triangulation.}
  \label{fig:lshape2}
  \end{center}
\end{figure}

\subsection{Weakly-singular integral equation with adaptive
mesh-refinement}\label{sec:examples:slp}
In our first experiment, we underline the result of
Theorem~\ref{thm:precond:slp}, which states the uniform boundedness of the
condition number of the preconditioned simple-layer operator.
We consider the homogeneous Laplace equation
\begin{subequations}\label{eq:laplace:dirichlet}
\begin{align}
  -\Delta u &=0 \quad\text{in }\Omega, \\
  u &= g \quad\text{on }\Gamma :=\partial\Omega
\end{align}
\end{subequations}
with given Dirichlet data $g\in H^{1/2}(\Gamma)$ and 
the L-shaped domain $\Omega$ from Figure~\ref{fig:lshape2}.
We note that $\diam(\Omega) = 2/3<1$.
Problem~\eqref{eq:laplace:dirichlet} is equivalent to the weakly-singular
integral equation
\begin{align}\label{eq:weaksing}
  \slp \phi = (1/2+\dlp)g,
\end{align}
where $\phi = \partial_{\normal} u$ and $(\dlp-1/2) : H^{1/2}(\Gamma) \to
H^{1/2}(\Gamma)$ denotes the trace of the double layer potential
\begin{align*}
  \widetilde\dlp g (x) = \int_\Gamma \partial_{\normal(y)} G(x-y) g(y)
  \,d\Gamma_y.
\end{align*}
Equation~\eqref{eq:weaksing} reads in the variational formulation: Find $\phi\in
\YY:=H^{-1/2}(\Gamma)$ such that
\begin{align}\label{eq:weaksing:varform}
  \dual{\psi}\phi_\slp = \dual{\psi}{(1/2+\dlp)g}_\Gamma \quad\text{for all
  }\psi \in\YY.
\end{align}
We prescribe the exact solution 
\begin{align*}
  u(x,y) = r^{2/3} \cos(2\varphi/3)
\end{align*}
with $(x,y)=(r\cos\varphi,r\sin\varphi)$ given in 2D polar coordinates. Then, $g := u|_\Gamma$ and
$\phi = \partial_{\normal} u$. The exact solution of~\eqref{eq:weaksing:varform}
exhibits a generic singularity at the reentrant corner $(0,0)\in\R^2$.
We use the local ZZ-type error indicators developed in~\cite{zzbem} to steer
the mesh-adaptation and to resolve this singularity effectively.

The discrete version of~\eqref{eq:weaksing:varform} reads in matrix notation:
Find $\pphi \in \R^{M_L}$ such that
\begin{align}
  \AA_\slp \pphi = \GG \in\R^{M_L},
\end{align}
where $(\GG)_j = \dual{\psi_{T_j}}{(1/2+\dlp)g}_\Gamma$ for all $j=1,\dots,M_L$.
Due to~\cite{amt99}, the $\ell_2$-condition number of the Galerkin matrix
$\AA_\slp\in\R^{M_L\times M_L}$ is
bounded by
\begin{align}\label{eq:slp:cond2}
  \cond_2(\AA_\slp) \lesssim M_L
  \Big(\frac{\hmax}{\hmin}\Big)^2 (1+|\log(M_L\hmax)|) =:\alpha_L
\end{align}
and, thus, can become bad on adaptively refined meshes. Therefore, we consider
the preconditioned system
\begin{align}
  \PPrec^{-1} \AA_\slp \pphi = \PPrec^{-1}\GG
\end{align}
where the preconditioner matrix $\PPrec\in\R^{M_L \times M_L}$ is either the
local multilevel preconditioner $\PPrec_\slp$ proposed in
Section~\ref{sec:main:LMLD} or the simple diagonal scaling $\PPrec_{\rm diag} :=
\diag(\AA_\slp)$ proposed in~\cite{amt99}.
According to Theorem~\ref{thm:precond:slp}, the eigenvalues of $\PPrec_\slp^{-1}
\AA_\slp$ are uniformly bounded. Since $\PPrec_\slp^{-1}\AA_\slp$ is symmetric
with respect to $\dual\cdot\cdot_{\AA_\slp}$ and
$\dual\cdot\cdot_{\PPrec_\slp}$, the condition number can be estimated by
\begin{align}\label{eq:slp:condLMLD}
  \cond_{\AA_\slp}(\PPrec_\slp^{-1} \AA_\slp)=
  \cond_{\PPrec_\slp}(\PPrec_\slp^{-1} \AA_\slp) =
  \frac{\evmax(\PPrec_\slp^{-1}\AA_\slp)}{\evmin(\PPrec_\slp^{-1}\AA_\slp)}
  \lesssim 1
\end{align}
with $\evmin(\cdot)$ and $\evmax(\cdot)$ being the minimal resp.\ maximal eigenvalue.
On the other hand, it has been proved in~\cite{amt99} that
\begin{align}\label{eq:slp:condDiag}
  \cond_{\AA_\slp}(\PPrec_{\rm diag}^{-1}\AA_\slp) \lesssim M_L (1+|\log(M_L \hmin)|)
  \frac{1+|\log\hmin|}{1+|\log\hmax|} =: \beta_L.
\end{align}

\begin{figure}[t]
\begin{center}
  \psfrag{nE}[c][c]{\tiny number of elements $M_L$}
  \psfrag{condition numbers}{\tiny condition numbers}
  \psfrag{alphaL}{\tiny $\OO(\beta_L)$}
  \psfrag{betaL}{\tiny $\OO(\alpha_L)$}
  \psfrag{V}{\tiny $\cond_2(\AA_\slp)$}
  \psfrag{PV}{\tiny $\cond_{\AA_\slp}(\PPrec_\slp^{-1}\AA_\slp)$}
  \psfrag{PdiagV}{\tiny $\cond_{\AA_\slp}(\PPrec_{\rm diag}^{-1}\AA_\slp)$}
  \includegraphics[width=0.7\textwidth]{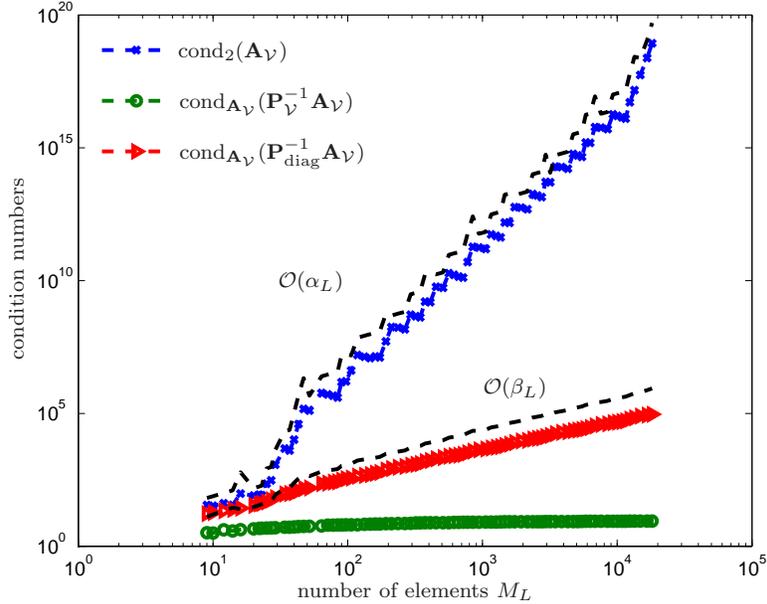}
  \caption{Condition numbers for the unpreconditioned matrix $\AA_\slp$ and the
    preconditioned matrices $\PPrec_\slp^{-1}\AA_\slp$ and $\PPrec_{\rm
    diag}^{-1} \AA_\slp$ for the weakly-singular integral equation from Section~\ref{sec:examples:slp}.
    Here, $\alpha_L$ resp. $\beta_L$ are the upper bounds in the
    estimates~\eqref{eq:slp:cond2} resp.~\eqref{eq:slp:condDiag}.}
  \label{fig:SLP:cond}
\end{center}
\end{figure}
\noindent
In Figure~\ref{fig:SLP:cond}, we compare the condition numbers of the Galerkin matrix
$\AA_\slp$ and the preconditioned matrices $\PPrec_\slp^{-1}\AA_\slp$ and
$\PPrec_{\rm diag}^{-1}\AA_\slp$.
We observe that the condition numbers behave as predicted by the
estimates~\eqref{eq:slp:cond2}
and~\eqref{eq:slp:condLMLD}--\eqref{eq:slp:condDiag}.

\subsection{Transmission problem with adaptive mesh-refinement}\label{sec:examples:jn}
Let $\Omega$ denote the L-shaped domain from Figure~\ref{fig:lshape2}.
We consider the (stabilized) Johnson-N\'ed\'elec FEM-BEM
coupling~\eqref{eq:varform:stabilized} for
the transmission problem~\eqref{eq:model} with $\material(x)$ being
the $2\times 2$ identity matrix, i.e. $-\div(\material\nabla u) = -\Delta u$ in
$\Omega$.
We prescribe the exact solutions
\begin{align}
  u(x,y) &= r^{2/3} \cos(2\varphi/3) \quad\text{for }(x,y)\in\Omega, \\
  u^{\rm ext}(x,y) &= \frac1{10} \frac{x+y-0.125}{(x-0.125)^2 + y^2} \quad\text{for
  }(x,y)\in\Omegaext,
\end{align}
where $(r,\varphi)$ denote the 2D polar coordinates.
The data $f\in L^2(\Omega)$, $u_0\in H^{1/2}(\Gamma)$, and $\phi_0\in
H^{-1/2}(\Gamma)$ are computed thereof.
We stress that $u$, hence also $u_0 = (u-u^{\rm ext})|_\Gamma$, exhibits a
generic singularity at the reentrant corner $(0,0)\in\R^2$.
To steer the mesh-adaptivity, we use the residual-based error estimator from~\cite{affkmp,afkp}
which dates back to~\cite{cs1995} for the symmetric coupling.

\begin{figure}[t]
\begin{center}
  \psfrag{nE}[c][c]{\tiny number of elements $\# \TT_L^\Omega$}
  \psfrag{condition numbers}{\tiny condition numbers}
  \psfrag{condJN}{\tiny estimate for $\cond_2(\AA_\jn)$}
  \psfrag{condBprec}{\tiny $\cond_{\PPrec_\jn}(\PPrec_\jn^{-1}\AA_\pform)$}
  \psfrag{condBprecHB}{\tiny $\cond_{\PPrec_\jn^{\rm HB}}( (\PPrec_\jn^{\rm HB})^{-1}\AA_\pform)$}
  \includegraphics[width=0.7\textwidth]{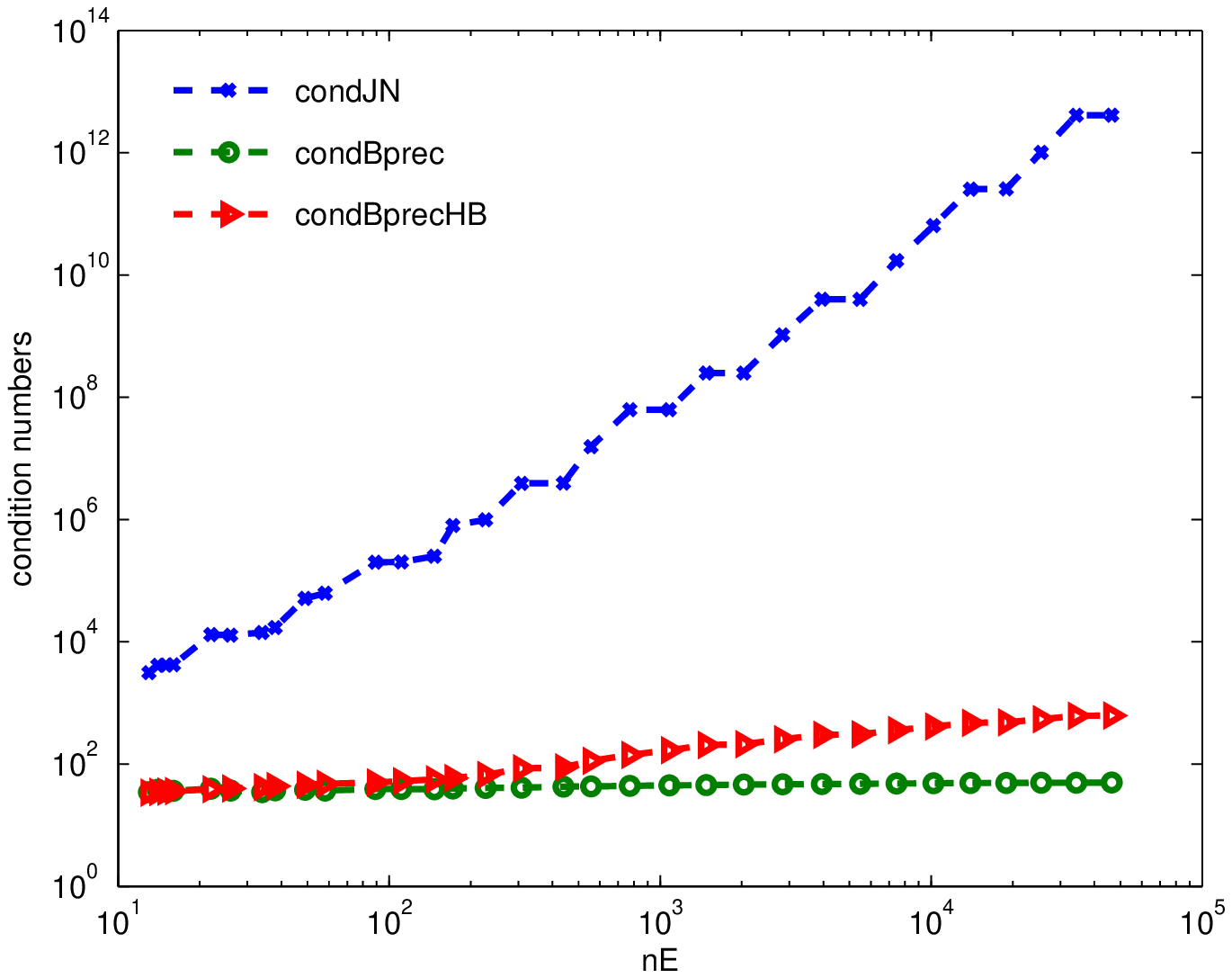}
  \caption{Estimate~\eqref{eq:cond2:estimate} for the condition number of $\AA_\jn$ and the condition
  number of $\PPrec_\jn^{-1}\AA_\pform$ resp. $(\PPrec_\jn^{\rm
  HB})^{-1}\AA_\pform$ for the stabilized Johnson-N\'ed\'elec coupling of Section~\ref{sec:examples:jn}.}
  \label{fig:JN:cond}
\end{center}
\end{figure}

\begin{figure}[t]
\begin{center}
  \psfrag{nE}[c][c]{\tiny number of elements $\# \TT_L^\Omega$}
  \psfrag{number of iterations}{\tiny number of iterations}
  \psfrag{stab}{\tiny \emph{stabilized}, $\PPrec = \PPrec_\jn$}
  \psfrag{stabHB}{\tiny \emph{stabilized}, $\PPrec = \PPrec_\jn^{\rm HB}$}
  \psfrag{nostab}{\tiny \emph{non-stabilized}, $\PPrec = \PPrec_\jn$}
  \psfrag{nostabHB}{\tiny \emph{non-stabilized}, $\PPrec = \PPrec_\jn^{\rm HB}$}
  \includegraphics[width=0.7\textwidth]{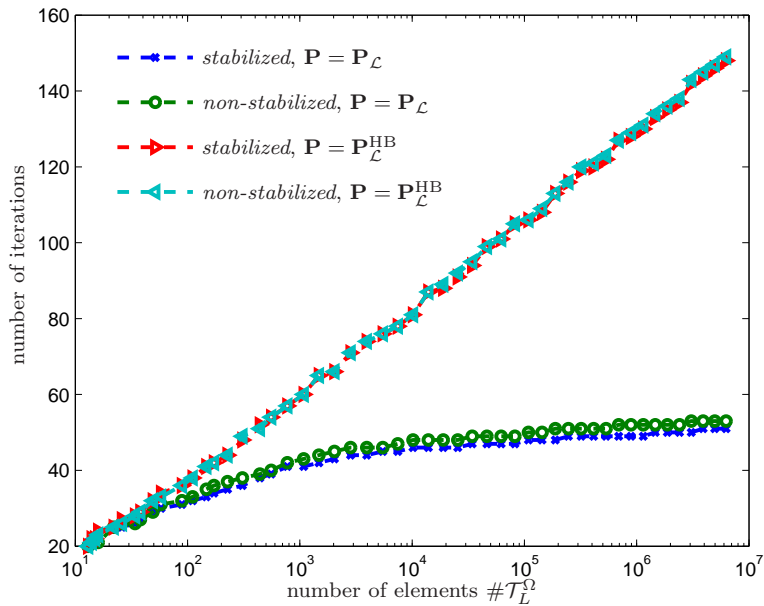}
  \caption{Number of iterations for solving the
  non-stabilized~\eqref{eq:jn:nonstab} resp. stabilized Johnson-N\'ed\'elec
  coupling~\eqref{eq:jn:stab} using the preconditioned GMRES Algorithm~\ref{alg:gmres} with tolerance $\tau
  = 10^{-6}$, inner product $\PPrec = \PPrec_\jn$ resp. $\PPrec =
  \PPrec_\jn^{\rm HB}$, and initial guess $\UU_0=\bignull$.}
  \label{fig:JN:iter}
\end{center}
\end{figure}

\begin{figure}[htb]
\begin{center}
  \psfrag{nE}[c][c]{\tiny number of elements $\# \TT_L^\Omega$}
  \psfrag{number of iterations}{\tiny number of iterations}
  \psfrag{jnAdap}{\tiny $\PPrec_\jn^{-1}\AA_\jn$, adap.}
  \psfrag{jnUnif}{\tiny $\PPrec_\jn^{-1}\AA_\jn$, unif.}
  \psfrag{symAdap}{\tiny $\PPrec_\sym^{-1}\AA_\sym$, adap.}
  \psfrag{symUnif}{\tiny $\PPrec_\sym^{-1}\AA_\sym$, unif.}
  \includegraphics[width=0.7\textwidth]{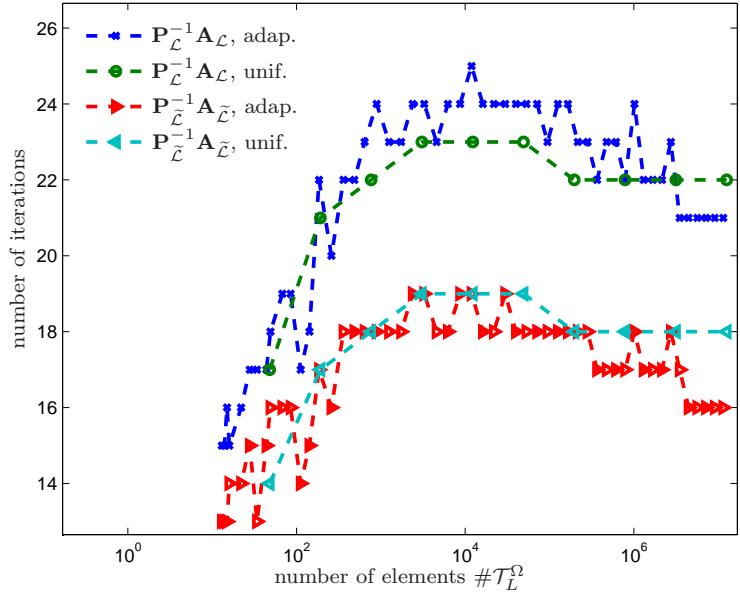}
  \caption{Number of iterations for solving the stabilized
  Johnson-N\'ed\'elec coupling~\eqref{eq:jn:stab} resp. symmetric 
  coupling~\eqref{eq:sym:precondsystem} using the preconditioned GMRES Algorithm~\ref{alg:gmres} with tolerance $\tau
  = 10^{-3}$ and inner product $\PPrec = \PPrec_\jn$ resp. $\PPrec =
  \PPrec_\sym$ on adaptively and uniformly refined meshes. For the
  initial guess $\UU_0$ we prolongate the solution of~\eqref{eq:jn:stab}
  resp.~\eqref{eq:sym:precondsystem} at level $L-1$ to level $L$.}
  \label{fig:JNSym:iter}
\end{center}
\end{figure}

For the Johnson-N\'ed\'elec coupling, we compare the proposed optimal preconditioner $\PPrec_\jn$ with the
block-diagonal preconditioner
\begin{align}
  \PPrec_\jn^{\rm HB} = \begin{pmatrix}
    \PPrec_\fem^{\rm HB} & \bignull \\
    \bignull & \PPrec_\slp^{\rm HB}
  \end{pmatrix},
\end{align}
which was proposed and analyzed in~\cite{ms98} for the symmetric coupling. 
Here, $\PPrec_\fem^{\rm HB}$ denotes the hierarchical basis preconditioner
corresponding to the operator $\fem$, and $\PPrec_\slp^{\rm HB}$ denotes the
hierarchical basis preconditioner corresponding to the simple-layer operator
$\slp$.
Basically, the difference between local multilevel preconditioners and
hierarchical preconditioners is that the set $\widetilde\NN_\ell^\Omega$
resp. $\widetilde\NN_\ell^\Gamma$ is replaced by the set of new nodes
$\NN_{\ell+1}^\Omega\backslash \NN_\ell^\Omega$ resp. $\NN_{\ell+1}^\Gamma
\backslash \NN_\ell^\Gamma$. This means that scaling is only done on the newly
created nodes, but not on their neigbours. 
It is well-known that hierarchical basis preconditioners lead to
sub-optimal condition number, which depend on the number of levels $L$.
A more detailed discussion can be found in~\cite{yserentant} for FEM problems
and in~\cite{tsm97} for BEM model problems.
See also~\cite[Section~6]{xch10} resp.~\cite[Section~3]{ffps} for a numerical comparison between hierarchical
basis and local multilevel additive Schwarz preconditioners for some FEM resp. BEM problems on adaptively
refined meshes.
Sub-optimality of $\PPrec_\fem^{\rm HB}$ and $\PPrec_\slp^{\rm
HB}$ lead to sub-optimality of the FEM-BEM preconditioner $\PPrec_\jn^{\rm HB}$,
i.e. a dependency on the level $L$. Thus, also the number of iterations depend
on $L$, which is also seen in our numerical examples.

In Figure~\ref{fig:JN:cond}, we plot
$\cond_{\PPrec_\jn}(\PPrec_\jn^{-1}\AA_\pform)$ which is an upper bound for
$\cond_{\PPrec_\jn}(\PPrec_\jn^{-1}\AA_\jn)$, see
  Remark~\ref{rem:cond},
and compare it with the condition
number $\cond_{\PPrec_\jn^{\rm HB}}( (\PPrec_\jn^{\rm HB})^{-1} \AA_\pform)$. We
observe that the condition number of the preconditioned matrix $(\PPrec_\jn^{\rm
HB})^{-1}\AA_\jn$ depends on the level $L$, whereas the condition number of
$\PPrec_\jn^{-1}\AA_\jn$ is independent of the level $L$. This underlines the
optimality of the preconditioner $\PPrec_\jn$ as stated in
Theorem~\ref{thm:main}.
Additionally, we plot the estimate
\begin{align}\label{eq:cond2:estimate}
  \cond_2(\AA_\jn) \leq \sqrt{\cond_1(\AA_\jn)\cond_1(\AA_\jn^T)}.
\end{align}
for the $\ell_2$-condition number of $\AA_\jn$ in Figure~\ref{fig:JN:cond}.
An estimate for the condition number $\cond_1(\AA_\jn) = \norm{\AA_\jn}1 \norm{\AA_\jn^{-1}}1$
is computed with the {\sc Matlab} function {\tt condest}.
Estimate~\eqref{eq:cond2:estimate} is obtained from
\begin{align*}
  \norm{\AA}2 \leq \sqrt{\norm{\AA}1\norm{\AA}\infty} \quad\text{for all
  matrices }\AA.
\end{align*}

In Figure~\ref{fig:JN:iter}, we furthermore consider the \emph{non-stabilized} system
\begin{align}\label{eq:jn:nonstab}
  \PPrec^{-1} \widehat\AA_\jn \UU = \PPrec^{-1}\widehat\FF,
\end{align}
where $\widehat\AA_\jn$ corresponds to the Galerkin matrix of the
\emph{non-stabilized} problem~\eqref{eq:jn:disrete} and $\widehat\FF$
corresponds to the right-hand side of~\eqref{eq:jn:disrete}.
The matrix $\PPrec$ is either the preconditioner matrix $\PPrec_\jn$ or
$\PPrec_\jn^{\rm HB}$.
Note that by Lemma~\ref{lem:jn:stab}, the solution $\UU$
of~\eqref{eq:jn:nonstab} is unique and also a solution of
\begin{align}\label{eq:jn:stab}
  \PPrec^{-1} \AA_\jn \UU = \PPrec^{-1}\FF.
\end{align}
In Figure~\ref{fig:JN:iter}, we plot the number of iterations used in the
preconditioned GMRES Algorithm~\ref{alg:gmres} with tolerance $\tau=10^{-6}$,
inner product $\PPrec = \PPrec_\jn$ resp. $\PPrec = \PPrec_\jn^{\rm HB}$, and initial guess $\UU_0=\bignull$ for solving the
problem~\eqref{eq:jn:stab} and problem~\eqref{eq:jn:nonstab}.
We observe that, both for $\PPrec_\jn$ and $\PPrec_\jn^{\rm HB}$, the number of iterations for solving the \emph{non-stabilized}
problem~\eqref{eq:jn:nonstab} is slightly higher than the number of iterations
used for solving problem~\eqref{eq:jn:stab} with the \emph{stabilized} system
matrix $\AA_\jn$.


\begin{figure}[htb]
\begin{center}
  \psfrag{x}[c][c]{\tiny $x$}
  \psfrag{y}[c][c]{\tiny $y$}
  \includegraphics[width=0.49\textwidth]{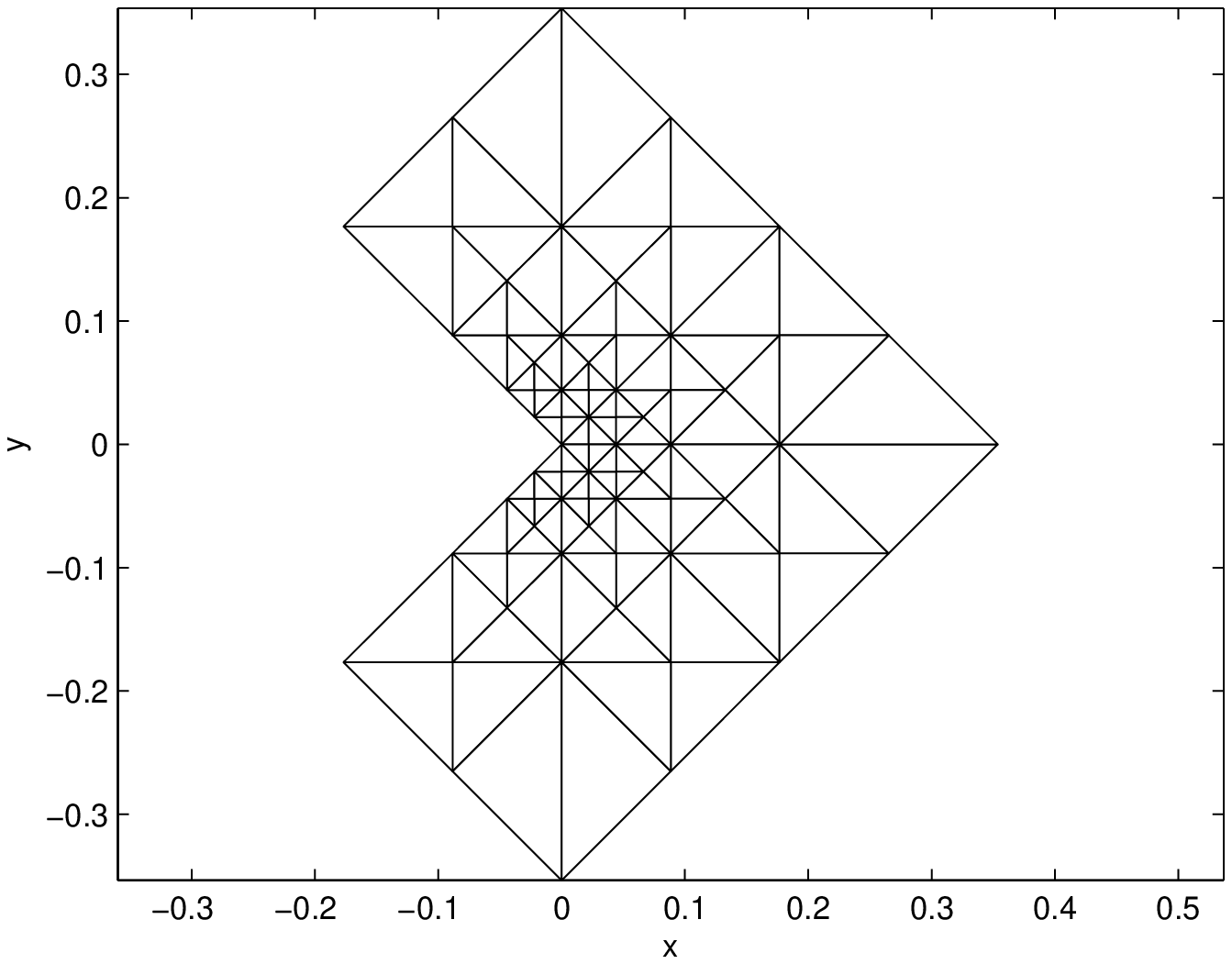}
  \includegraphics[width=0.49\textwidth]{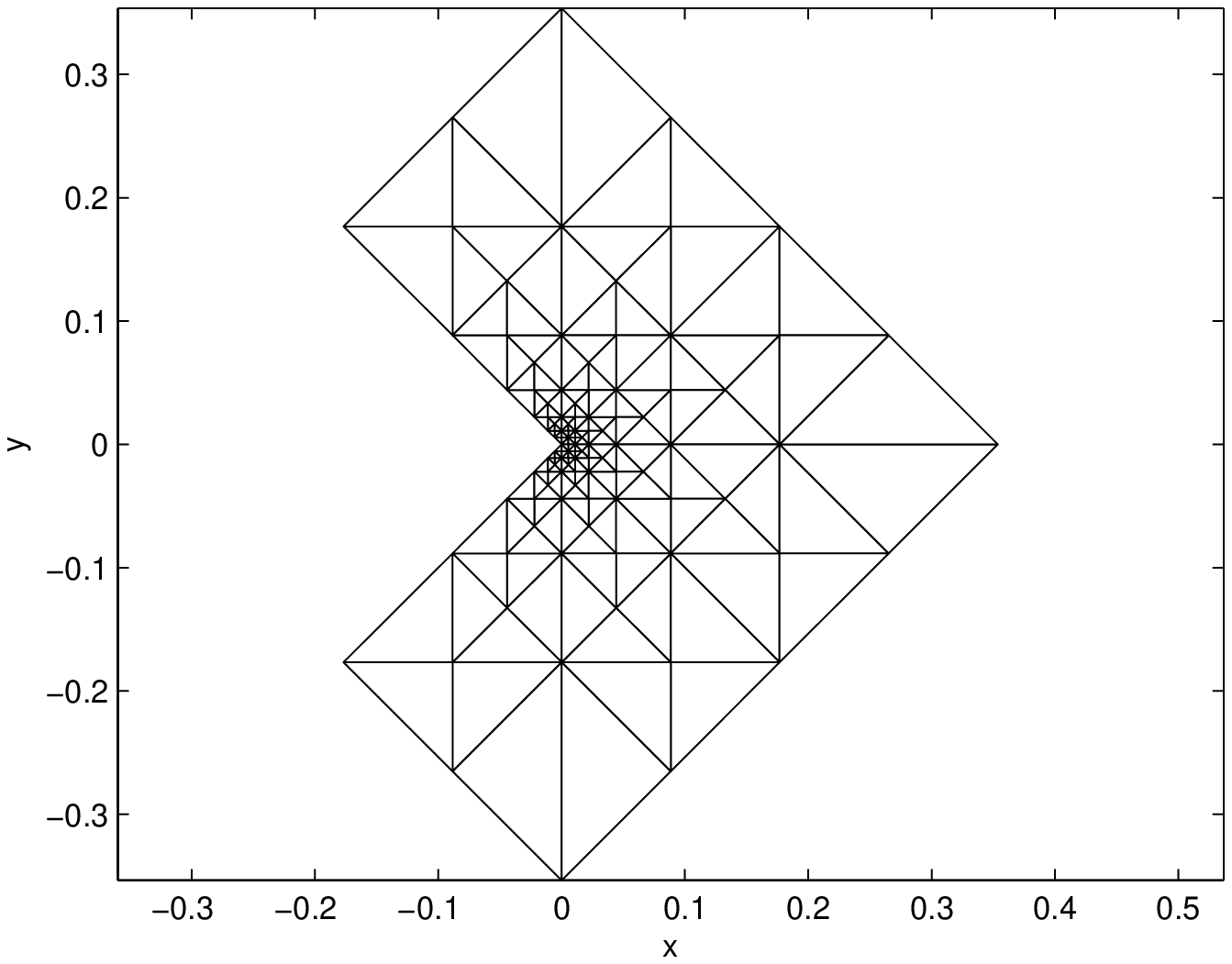}
  \caption{Artificially refined meshes at level $L=3$ (left) with $\#
    \TT_3^\Omega = 114$ resp. $M_3 = \# \TT_3^\Gamma = 20$ and level
    $L=5$ (right) with $\# \TT_5^\Omega = 186$ resp. $M_5 = \# \TT_5^\Gamma =24$
    for the stabilized Johnson-N\'ed\'elec coupling of Section~\ref{sec:examples:artRef}.}
  \label{fig:mesh:artificial}
\end{center}
\end{figure}

\begin{table}[!htb]
\begin{center}
  \begin{small}
    \vspace*{1ex}
  \begin{tabular}{|c|c|c|c|c|c|c|c|}
\hline
\hline $L$ & $\#\TT_L^\Omega$ & $M_L$ & $\cond_2(\AA_\jn)$ &
{\tiny $\cond_{\PPrec_\jn}(\PPrec_\jn^{-1}\AA_\pform)$} &
{\tiny $\cond_{\PPrec_\jn^{\rm HB}}((\PPrec_\jn^{\rm HB})^{-1}\AA_\pform)$} & $\hmax$ & $\hmin$ \\
\hline\hline
1 & 42 & 16 & 4.57e+03 & 36.65 & 43.72 & 0.18 & 8.84e-02\\\hline
2 & 78 & 18 & 1.60e+04 & 41.03 & 46.69 & 0.18 & 4.42e-02\\\hline
3 & 114 & 20 & 6.11e+04 & 44.47 & 49.49 & 0.18 & 2.21e-02\\\hline
4 & 150 & 22 & 2.40e+05 & 47.26 & 54.62 & 0.18 & 1.10e-02\\\hline
5 & 186 & 24 & 9.54e+05 & 49.39 & 67.47 & 0.18 & 5.52e-03\\\hline
6 & 222 & 26 & 3.80e+06 & 51.03 & 86.40 & 0.18 & 2.76e-03\\\hline
7 & 258 & 28 & 1.52e+07 & 52.30 & 108.98 & 0.18 & 1.38e-03\\\hline
8 & 294 & 30 & 6.07e+07 & 53.30 & 134.68 & 0.18 & 6.91e-04\\\hline
9 & 330 & 32 & 2.43e+08 & 54.10 & 163.37 & 0.18 & 3.45e-04\\\hline
10 & 366 & 34 & 9.70e+08 & 54.75 & 195.00 & 0.18 & 1.73e-04\\\hline
11 & 402 & 36 & 3.88e+09 & 55.28 & 229.54 & 0.18 & 8.63e-05\\\hline
12 & 438 & 38 & 1.55e+10 & 55.72 & 267.00 & 0.18 & 4.32e-05\\\hline
13 & 474 & 40 & 6.21e+10 & 56.09 & 307.35 & 0.18 & 2.16e-05\\\hline
14 & 510 & 42 & 2.48e+11 & 56.40 & 350.59 & 0.18 & 1.08e-05\\\hline
15 & 546 & 44 & 9.93e+11 & 56.67 & 396.72 & 0.18 & 5.39e-06\\\hline
16 & 582 & 46 & 3.97e+12 & 56.89 & 445.73 & 0.18 & 2.70e-06\\\hline
17 & 618 & 48 & 1.59e+13 & 57.09 & 497.62 & 0.18 & 1.35e-06\\\hline
18 & 654 & 50 & 6.35e+13 & 57.26 & 552.40 & 0.18 & 6.74e-07\\\hline
19 & 690 & 52 & 2.54e+14 & 57.41 & 610.05 & 0.18 & 3.37e-07\\\hline
20 & 726 & 54 & 1.02e+15 & 57.54 & 670.57 & 0.18 & 1.69e-07\\\hline
21 & 762 & 56 & 4.07e+15 & 57.66 & 733.97 & 0.18 & 8.43e-08\\\hline
22 & 798 & 58 & 1.63e+16 & 57.76 & 800.24 & 0.18 & 4.21e-08\\\hline
23 & 834 & 60 & 6.51e+16 & 57.85 & 869.38 & 0.18 & 2.11e-08\\\hline
\end{tabular}

    \vspace*{1ex}
  \end{small}
  \caption{Condition number estimates for the example with artificial
    mesh-refinement from Section~\ref{sec:examples:artRef}.
    Note that $\cond_{\PPrec_\jn}(\PPrec_\jn^{-1}\AA_\pform)$ is an
    upper bound for the condition number
    $\cond_{\PPrec_\jn}(\PPrec_\jn^{-1}\AA_\jn)$, see Remark~\ref{rem:cond}, and    
    $\cond_{\PPrec_\jn^{\rm HB}}((\PPrec_\jn^{\rm HB})^{-1}\AA_\pform)$
    is an upper bound for $\cond_{\PPrec_\jn^{\rm HB}}((\PPrec_\jn^{\rm
    HB})^{-1}\AA_\jn)$.  }
  \label{tab:artRef:cond}
\end{center}
\end{table}

\subsection{Symmetric coupling vs. Johnson-N\'ed\'elec coupling}
In a further experiment, we compare the (stabilized)
Johnson-N\'ed\'elec coupling~\eqref{eq:jn:stab} and the (stabilized) symmetric
coupling~\eqref{eq:sym:precondsystem}
with respect to the number of iterations used in the preconditioned GMRES
Algorithm~\ref{alg:gmres} with $\tau = 10^{-3}$ and $\PPrec
= \PPrec_\jn$ resp. $\PPrec = \PPrec_\sym$. For the initial guess
$\UU_0$ we prolongate the solution of~\eqref{eq:jn:stab} resp.~\eqref{eq:sym:precondsystem} at level
$L-1$ to level $L$.
Mesh-adaptivity is steered with the solution of~\eqref{eq:sym:precondsystem} and the
residual-based error estimator from~\cite{cs1995}.
In Figure~\ref{fig:JNSym:iter}, we plot the number of iterations used for
adaptive refinement. We observe that for both the uniform and adaptive case, the
symmetric coupling needs less iterations.
However, the symmetric coupling requires the computation of additional
matrix-vector multiplications with discrete BEM operators in each iteration
step.

\subsection{Transmission problem with artificial refinement}\label{sec:examples:artRef}
Let $\Omega$ denote the L-shaped domain with boundary $\Gamma = \partial\Omega$
and initial triangulations $\TT_0^\Omega$, $\TT_0^\Gamma$ from 
Figure~\ref{fig:lshape2}.
We consider an artificial mesh-refinement, where we only mark the elements
$T\in\TT_\ell^\Omega$, $\TT_\ell^\Gamma$ with $(0,0)\in T$ for refinement.
Clearly, this leads to strongly adapted meshes towards the origin
$(0,0)\in\R^2$, see Figure~\ref{fig:mesh:artificial}.
As for the example from Section~\ref{sec:examples:jn}, we compare
$\cond_{\PPrec_\jn}(\PPrec_\jn^{-1}\AA_\pform)$ and $\cond_{\PPrec_\jn^{\rm HB}}(
(\PPrec_\jn^{\rm HB})^{-1}\AA_\pform)$ as well as the
estimate~\eqref{eq:cond2:estimate} for $\cond_2(\AA_\jn)$.
The results are summarized in Table~\ref{tab:artRef:cond}. We observe optimality of the
proposed preconditioner $\PPrec_\jn$, whereas the condition numbers for
the hierarchical preconditioner depend on the number of levels $L$.

\bibliographystyle{alpha}
\bibliography{literature}

\end{document}